\documentclass{amsart} 
\usepackage{latexsym,amsmath, amscd,graphicx, listings,multicol}
\newtheorem{thm}{Theorem}[section]

\newtheorem{prop}[thm]{Proposition}
 
\newtheorem{deff}[thm]{Definition} 
\newtheorem{rem}[thm]{Remark}

\newcommand{\BZ}{{\mathbb{Z}}}

\newcommand{\BC}{{\mathbb{C}}}
\newcommand{\BQ}{{\mathbb{Q}}}

\newcommand{\CE}{{\mathcal{E}}}

 \newcommand{\CL}{{\mathcal{L}}}
\newcommand{\CI}{{\mathcal{I}}}

\newcommand{\CN}{{\mathcal{N}}}
\newcommand{\CQ}{{\mathcal{Q}}}

\newcommand{\fE}{{\mathfrak{E}}}

\newcommand{\al}{{\alpha}}
\newcommand{\be}{{\beta}}
\newcommand{\de}{{\delta}}
\newcommand{\ga}{{\gamma}}

\newcommand{\si}{{\sigma}}

\newcommand{\ep}{\epsilon}

\newcommand{\De}{{\Delta}}
\newcommand{\Ga}{{\Gamma}}

\newcommand{\fe}{{\mathfrak{e}}}

\newcommand{\fc}{{\mathfrak{c}}}

\newcommand{\fv}{{\mathfrak{v}}}
\newcommand{\fw}{{\mathfrak{w}}}

\newcommand{\fZ}{{\mathfrak{Z}}}
\newcommand{\fJ}{{\mathfrak{J}}}

\DeclareMathOperator{\Int}{Int}
\DeclareMathOperator{\Lk}{Lk}
\DeclareMathOperator{\Sign}{Sign}
\DeclareMathOperator{\SignE}{Sign_{\CE}}
\DeclareMathOperator{\nullity}{nullity}
\DeclareMathOperator{\sig}{sig}
\DeclareMathOperator{\nul}{nul}

\DeclareMathOperator{\pari}{parity}

\newcommand{\rp}{{\mathbb{R}}P(2)}
\newcommand{\cp}{{\mathbb{C}}P(2)}

\begin{document}

\title{Signatures of real algebraic curves via plumbing diagrams}

\author{ Patrick M. Gilmer}
\address{Department of Mathematics\\
Louisiana State University\\
Baton Rouge, LA 70803\\
USA}
\email{gilmer@math.lsu.edu}
  \urladdr{www.math.lsu.edu/\textasciitilde gilmer/}

\author{ Stepan Yu. Orevkov}
\address{IMT, Univertit\'e Paul Sabatier,
118 route de Narbonne,
31061 Toulouse, France
and \\
Steklov Math. Inst.\\ Moscow, Russia\\
and
National Research University Higher School of Economics\\
Russian Federation. AG Laboratory, HSE\\
Vavilova 7\\Moscow, Russia,
117312}
\email{orevkov@picard.ups-tlse.fr }
\urladdr{http://www.math.univ-toulouse.fr/~orevkov/}

\thanks{The first author was partially supported by  NSF-DMS-0905736,  NSF-DMS-1311911. 
The second author was partially supported by
RSF grant, project 14-21-00053 dated 11.08.14} 

\date{November 24, 2017}

\begin{abstract} We define and calculate signature and nullity invariants for complex schemes for curves in $\rp$. We use 
an analog of
the Murasugi-Tristram inequality to prohibit certain schemes from being realized by real algebraic curves.
We  give new formulas for  Casson-Gordon invariants of graph manifolds, and  signatures of graph links
\end{abstract}

\maketitle
\setcounter{tocdepth}{2}

\section{Introduction} \label{sec.intro} 
\subsection{Real algebraic curves  and associated links}\label{DG}
A nonsingular real algebraic curve in $\rp$ consists of a collection of disjoint simple smooth closed curves.
Such a collection whether coming  from a real algebraic curve or not consists of a number of 2-sided curves called ovals, and at most one $1$-sided curve. We  call 
the isotopy type of such a collection of curves a real scheme.  A real algebraic curve includes a 1-sided component if and only if it has odd degree. We say a scheme has odd or even type accordingly to whether the scheme includes a  1-sided component or does not include a 1-sided curve.  A  real algebraic curve is called dividing  
if it separates its complex locus in $\cp$ into two components. In this case, the real part inherits from the complex curve a semi-orientation which is an orientation of each component, up  to simultaneous reversal of all components. This is called a complex orientation of the real algebraic curve. A real scheme equipped with a semi-orientation is called a complex scheme.

Let $\tilde D$ be the complement in $\cp$ of a small open tubular neighborhood of
$\rp$ which is invariant with respect to complex conjugation. Let  $\tilde Q =\partial \tilde D$ and let $D$ and $Q$ be the quotient of $\tilde D$ and $\tilde Q$ respectively by 
complex conjugation. It is convenient to assume that the tubular neighborhood is
``infinitely small''. 
Rigorously speaking, this means that $\tilde D$ is $\cp$ (viewed as a
real variety) blown up at $\rp$ and then cut along the exceptional divisor. Then
$Q$ is the projective normal bundle of $RP(2)$ which we identify with the projective
tangent bundle of $\rp$ by sending each tangent vector $v$ to $iv$, where $i =
\sqrt {-1}$.  
Under this identification, $D$ is the projective
tangent disk bundle of $\rp$.

 Let $A$ be an embedded  surface in $\cp$, 
 invariant under complex conjugation
 and such that the tangent bundle to $A$ restricted to $C$ is the complexification of the tangent bundle to $C$ where $C = A \cap \rp$. Then $A \cap \tilde Q$  is a link in $\tilde Q$ which is
determined by $C$.  We denote its orbit in $Q$ under complex conjugation by $L(C)$. 
One may describe $L(C)$ more directly as follows:
To each point $x$ of $C$, we have the tangent line $l_x$ to $C$ at $p$. The pair $(x,l_x)$ is a point in $Q$ the projective tangent bundle of $\rp$. 
One defines $L(C)= \{  (x,l_x)| x \in C\}$. It is a link in $Q$ with one component lying over each component of $C$ in $\rp$. 
A semi-orientation on $C$ induces a
semi-orientation on $L(C)$.   
{\it  We fix a single orientation on $C$ compatible with its semi-orientation, and take the associated orientation of $L(C)$.  This does not cause any problem as the signatures and nullities that we study are preserved if we reverse the orientation on all components of a link.}

The complex locus of a real algebraic curve (which we may assume is nonsingular as a complex curve) is  an embedded surface 
of genus  $(m-1)(m-2)/2$.
Thus a necessary condition for realizability of a complex
scheme $C$ by a real algebraic dividing curve of degree $m$ is that $L(C)$ bounds a
connected orientable surface $G$ (the orbit space of $A$ under complex orientation) properly embedded in $D$ with $\beta_1(G) = (m-1)(m-2)/2$.  
This condition is also necessary
for the
realizability
of $C$ by a flexible  dividing curve in the sense of Viro [V1].
{ \it All the restrictions on real algebraic curves that we discuss in this paper also hold for flexible curves. But, for brevity, we only mention this here.}

 In [G2], the first author defined signature and nullity invariants associated to $L(C)$. They were
 shown to satisfy an inequality deriving from an extension of the Murasugi-Tristram inequality for links in $S^3$. In this paper, we will describe $L(C)$ as a particularly simple link with respect to a plumbing diagram for $Q$ which is tailored to $C$.  Moreover, we will give formulas which allow the computation of the signature and nullity invariants of $L(C)$. In this way, we will prohibit certain complex schemes from arising as  complex orientations of dividing  real algebraic curves of specified degree. 
If we chose the opposite orientation on $C$ and $L(C)$, the signatures and nullities would remain the same.

Similar techniques can be applied to some related links that arise in $L(4,1)$ and $L(2,1)$  as covering spaces  of $Q$. This is work in progress.

\subsection{Nomenclature and numerical characteristics}
  
We collect here terminology and the definitions of numerical characteristics which are assigned to a scheme $C$.

\begin{itemize}
\item We use $l(C)$ to denote the number of ovals in $C$.

\item If $C$ has odd type,  we denote the 1-sided component  by  $J$.  If $C$ has even type,   we pick a one-sided curve and give it an arbitrary orientation. This auxiliary curve will also be denoted by $J$.  Thus we have a oriented one-sided curve $J$  in the complement of the ovals of $C$ which we can use as a reference. 

\item Abusing notation, we usually  let  $C$
indicate  either $C$ in the case $C$ has odd type, or $C$ together with a choice of $J$ in the case $C$ has even type. In this last case,
the isotopy class of the complex scheme $C \cup J$ usually depends on the choice of $J$. Thus if $C$ has even type,  further constructions  
depend on the  placement of  $J$. 
By  Theorem \ref{pplumb} below,  the signatures and nullities that we calculate for $C$ using 
$C \cup J$  are independent of the choice of $J$.

\item   An oval is said to be negative, if  the oval is homologous to twice the one sided curve in the complement of a point interior to the disk in the complement of the oval. Otherwise we say the oval is positive.  This counter-intuitive choice of words is traditional in the subject of real algebraic curves.  For an oval $o$ of $C$, let $\ep(o)$ be $1$ (resp. $-1$) if the oval is positive (resp. negative).

\item We will call the region of the complement  of $C \cup J$ which meets $J$ the outer region and denote it by $R_1$.
The other components of the complement  of $C \cup J$ will also be called regions.

\item If one must cross an even number of other ovals  to pass from a region to $R_1$, we say the region  is even and otherwise we say an  region is odd. 
We will refer to this evenness or oddness of a region $R$ as the parity of the region, and write $\pari(R) \in \{0,1\}$ accordingly. 

\item The ovals which meet $R_1$ are called outer ovals.

\item If one must cross an even number of other ovals  to pass from an oval to $R_1$, we say the oval  is even and otherwise we say an  oval is odd. 
We will refer to this evenness or oddness of an oval $o$ as the parity of an oval, and write $\pari(o) \in \{0,1\}$ accordingly. 

\item We let $\Lambda^+(C)$ ($\Lambda^-(C)$) denote the number of positive (respectively negative) ovals of $C$.

\item  Two ovals are said to form an injective pair if they are the boundary of an annulus in $\rp$. An injective pair of oriented ovals is said to be  positive (negative) if their orientation is  (respectively not) induced from an orientation on the annulus they bound. 
We let $\Pi^+(C)$ ($\Pi^-(C)$) denote the number of positive (respectively negative) injective pairs of ovals in $C$. 
\item For each region $R$, we let $\Lambda^+_R(C)$ (resp. $\Lambda^-_R(C)$) denote the number of positive (resp. negative) ovals of $C$ which encircle $R$ (or equivalently encircle an arbitrarily chosen point in $R$). 

\item For each oval  $o$, we let $\Pi^+_o(C)$ (resp. $\Pi^-_o(C)$) denote the number of ovals of $C$
which form positive (resp. negative) injective pairs with $o$.

\end{itemize}

We frequently omit  $C$ from the above notations when it should cause no confusion.

\subsection{Plumbing description of $L(C)$}
We begin by giving a plumbing description of the links $L(C)$. This description is inspired by the graph links of Eisenbud and Neumann \cite{EN}. The theory of graph links was  developed only for links in  integral homology spheres whereas our manifold, $Q$, is a rational homology sphere. Our particular case is  elementary and can be explained straight forwardly using Kirby calculus diagrams \cite{K}.

The  decorated weighted graph $\Gamma(C)$  for the plumbing diagram will have one vertex for each region $R$ of the complement  of $C \cup J$.  Such a vertex will be denoted by $R$, as well and be given, as weight, twice the Euler characteristic of $R$. For each oval $o$ in $C$, we have another vertex,  also denoted $o$, and weighted zero.  An edge connects $R$ and $o$ whenever $o$ is in the boundary of $R$. We have three more vertices $u_1$, $u_2$ and $u_3$ weighted   $1$, $2$ and $2$ in the graph. There are also three more edges one connecting $u_1$ to $u_2$, one connecting $u_1$ to $u_3$ and one connecting $u_1$ to the vertex associated to $R_1$.

Then we further decorate the graph with  signed arrows, one for each component of $C$.
If $C$ contains a one-sided curve, we attach a positive arrow to $u_2$. Otherwise we do not attach any arrow to $u_2$.  Now we can  add  signed arrows to all the vertices that are indexed by ovals. The sign of the arrow at the vertex indexed by  $o$ is given by 
$(-1)^{\pari(o)+1}\ep(o)$.

  \begin{figure}[h]
\includegraphics[width=2in]{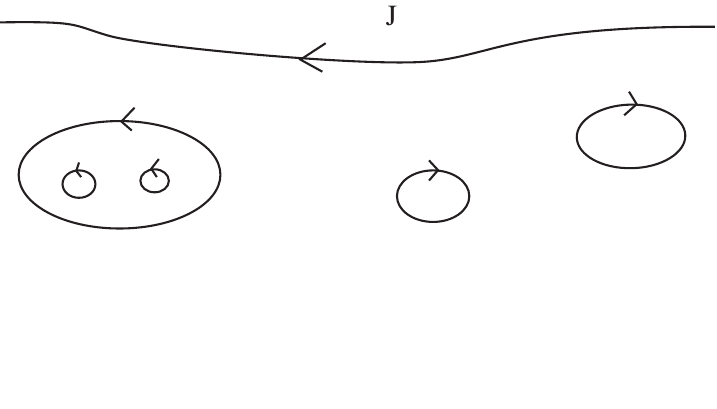} \quad \includegraphics[width=2in]{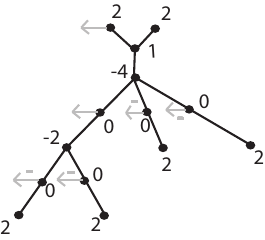}
\caption{On the left is a scheme $C$  of odd type  with five ovals and a 1-sided curve. On the right is the associated decorated graph $\Gamma(C)$. The six arrows represent the components of $L(C) \subset Q.$ The arrows are  assumed to have  the sign $+$, unless they are marked with a minus sign.   The weighted graph  is a plumbing diagram for $Q$. This complex scheme was chosen simply to illustrate some features of the construction of $\Gamma(C)$  and still not be too complicated.
This  complex scheme  cannot be the scheme for a dividing  real algebraic curve   as it does not satisfies the Rohlin-Mishachev restriction \cite[2.4]{R}  for a dividing curve of any degree.} \label{example}
\end{figure}

The plumbing diagram with the weights but without the signed arrows is a recipe to build a 4-manifold by gluing together oriented 2-disk bundles over 2-sphere bases \cite{HNK}. These 2-sphere bases are represented by the vertices. The weights are the Euler numbers of the  2-disk bundles. The edges, which form a tree in our case, describe the plan for plumbing these bundles together.  The boundary of this 4-manifold is homeomorphic to $Q$. The positive arrows depict oriented fibers of the associated 
circle bundle.  The negative arrows depict these  fibers with the opposite orientation. In both cases, these fibers are to be taken over points of the 2-disks which are the bases for the copies of $D^2 \times D^2$  used in the plumbing. Thus these arrows describe an oriented link $\CL(\Ga (C))$ in  $Q$. The plumbing matrix for $\Ga(C)$ is the symmetric matrix with rows and columns indexed by the vertices of $\Ga(C)$ whose diagonal entires are the weights and whose off-diagonal entries are one exactly when  the two vertices are connected by an edge, and are zero, otherwise. We note that the number of vertices of $\Gamma(C)$  is $2 l(C) +4$.

\begin{thm}\label{pplumb} $L(C)$ and $\CL(\Ga( C))$ describe the same oriented link in $Q$. 
\end{thm}

\subsection{Signatures and nullity of complex schemes $C$}

In Definition \ref{defsig}, for every  complex scheme $C$ and every odd prime $p$, and every $0< b \le (p-1)/2$, we define  a signature, $\sig_{b/p}(C)\in \BZ$, and a nullity, $\eta_{p}(C)\in \BZ_{\ge 0}$.

\begin{prop}{\label{signulprop}}
\begin{equation} \sig_{b/p}(C) +\eta_{p}(C) = \beta_0(C)-1 \pmod{2} \label{snc}
\end{equation}
\begin{equation}0 \le \eta_{p}(C) \le \beta_0(C)-1 .  \label{nb}
\end{equation}
\end{prop}

Here $\beta_0(C)$ is the 0-th Betti number of $C$, i.e. the number of components of $C$.
The following theorem can be viewed as analog  of the Murasugi-Tristram inequalities.   
It is a rephrasing  of \cite[Theorem(9.7)]{G2}. 

\begin{thm}\label{bound}
If $C$ is the real part of a dividing real algebraic curve of degree $m$ with its complex orientation, 
$p$ is an odd prime and  $0< b \le (p-1)/2$,  then
$$|\sig_{b/p}(C)|+\eta_{p}(C) \le (m-1)(m-2)/2.$$
\end{thm}

\subsection{Algorithm to calculate $\sig_{b/p}(C)$ and  $\eta_{p}(C)$\label{al}}
We now describe an algorithm for the calculation of $\sig_{b/p}(C)$ and  $\eta_{p}(C)$ with input 
$\Ga( C)$.

If $x\in \BZ$, let $r_p(x)$ denote the unique integer congruent to $x$ modulo $p$ in the range $[0,p-1]$.  If $\vec x$ is instead a vector in $\BZ^n$, for some $n$, $r_p(\vec x)$  is defined in this same way but  entry by entry.
If a vector $\vec x$ has entries indexed by the vertices $\fv$ of a  graph, we let $x_\fv$ denote the $\fv$-coefficient of $x$. 
If $Y$ is a  graph, we let $v(Y)$ be the vertices of $Y$. If $Y$ is a plumbing graph (ie. a graph with the vertices weighted by integers),
let $A_Y$ be the matrix whose rows and columns are indexed by $v(Y)$ and that  has for entries on the diagonal the weights and  for entries  off the diagonal has a  $1$ exactly when two vertices are joined by an edge. We let $|S|$ denote the cardinality of a finite set $S$.

Every symbol that we introduce here depends on $C$ implicitly.

\begin{itemize}

\item Let $\Ga$ stand for $\Gamma(C)$.

 \item Let $\vec s$ be a row vector in $\BZ^{|v(\Ga)|}$ with entries indexed by the vertices  $v(\Ga)$:
\begin{equation*}
\vec s_\fv=
\begin{cases} 1 &\text{if $\fv$ is the tail of a positive arrow } \\
-1 &\text{if $\fv$ is the tail of a negative  arrow }\\
0 &\text{if $\fv$ is not the tail of any  arrow. }
\end{cases}
\end{equation*}.

More explicitly:
\begin{equation*}
\vec s_\fv=
\begin{cases} 
0 &\text{if $\fv=u_2$ and $C$ is of even type} \\
1 &\text{if $\fv=u_2$ and $C$ is of odd type} \\
0 &\text{if  $\fv$ is $u_1$, $u_3$, or a region} \\
(-1)^{\pari(o)+1}\ep(o) &\text{if $\fv$ is an oval  $o$. } 
\end{cases}
\end{equation*}

\item Let $\De= 2 \vec s {A_\Ga}^{-1} {\vec s\ }^t \in  \BZ.$

\item Let $\vec c= -2 \vec s  {A_\Ga}^{-1}  \in  \BZ^n.$

\item Let $\Gamma^+$ be the plumbing graph obtained by converting all arrows of $\Gamma(C)$ to edges and all arrowheads to vertices weighted zero. We note that the number of vertices of $\Gamma^+$ for  a curve of odd type is $3 l(C) +5.$
For a curve of even type, it is one less.

\begin{figure}[h]
\includegraphics[width=2in]{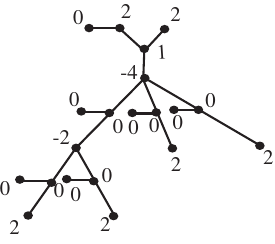} 
\caption{$\Ga^+$ for complex scheme on the left of Figure \ref{example} \label{ex4}} 
\end{figure}

\item Let $\vec c^+ $ be the vector in $\BZ^{|v(\Ga+)|}$indexed by  $v(\Ga^+)$ obtained by adjoining to the vector $\vec c$  extra  entries, indexed by the vertices which replace the arrowheads of the arrows, which are  $\pm 2$ according to the signs on the arrows.

\item Let $z_{\vec c, p}$ denote the number of the vertices $\fv$ of $\Ga$ with 
${\vec c}_{\fv}= 0\pmod{p}$. 

\item Let $\fZ_{\vec c, p}$ denote the subgraph of $\Ga$ whose vertices are  the  vertices $\fv$ of $\Ga$ such that ${\vec c}_{\fv}= 0\pmod{p}$ that are not connected by an edge of $\Ga^+$ to any vertex $\fw$ with ${\vec c^+}_{\fw} \ne 0\pmod{p}$, and whose edges are the edges of $\Ga$ connecting these vertices.

\item Let $e_{\vec c, p}$ denote the number of edges in $\Ga^+$  that join two vertices $\fv$ and $ \fw$ of $\Gamma^+$ with ${\vec c^+}_{\fv}\ne 0\pmod{p}$ and ${\vec c^+}_{\fw}\ne 0\pmod{p}$.

\item Let $\fE_{\vec c, p}$ be the number of edges $\fe$ in $\Ga^+$ with at least one endpoint 
$\fv$  of $\fe$ with ${\vec c^+}_{\fv}= 0\pmod{p}$.

 \end{itemize}

 \begin{thm}\label{form}  For every odd prime $p$, and every $1 \le b \le \frac {p-1}2$, we have that
 \begin{align*}
  \sig_{b/p}(C) &= \frac 2 {p^2} \left(r_p(b\vec c^+) ({A_{\Ga^+}}) (r_p(-b\vec c^+))^t+ b (p-2b)\Delta\right)+ \Sign(A_{\fZ_{\vec c, p}})-e_{\vec c, p} -2  \\ 
  \eta_{p}(C) &=  \fE_{\vec c, p}+ \nullity(A_{\fZ_{\vec c, p}}) + |v(\fZ_{\vec c, p})|-2 z_{\vec c, p}  . \end{align*}
 \end{thm}
 
 We implemented the above algorithm in the Mathematica program listed in Appendix \ref{sd}.
 For $C$ the complex scheme in Figure \ref{example} which in Viro's notation is 
$J \amalg 1^- \left<2^-\right> \amalg 2^+ $,
we  plot  $\sig_{b/127}(C)$ against $1 \le b\le 63$ in Figure \ref{ex}.   
  \begin{figure}[h]
\includegraphics[width=2in]{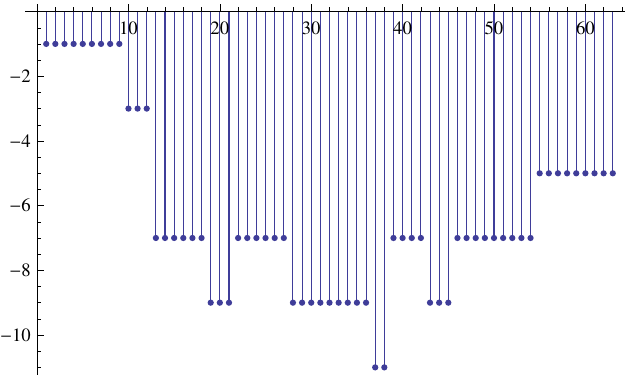} 
\caption{} \label{ex}
\end{figure}

Before we state the next theorem, we need the following definitions.
\begin{itemize}
\item Let $z_{\vec c,\infty}$ denote the number of the vertices $\fv$ of $\Ga$ with 
${\vec c}_{\fv}= 0$. 

\item Let $\fZ_{\vec c,\infty}$ denote the subgraph of $\Ga$ whose vertices are  the  vertices $\fv$ of $\Ga$ such that ${\vec c}_{\fv}= 0$ that are not connected by an edge of $\Ga^+$ to any vertex $\fw$ with ${\vec c^+}_{\fw} \ne 0$, and whose edges are the edges of $\Ga$ connecting these vertices.

\item Let $\fE_{\vec c,\infty}$ be the number of edges $\fe$ in $\Ga^+$ with at least one endpoint 
$\fv$  of $\fe$ with ${\vec c^+}_{\fv}= 0$.

\item Let $\nul(C)$ denote
 $ \fE_{\vec c,\infty}+ \nullity(A_{\fZ_{\vec c,\infty}}) + |v(\fZ_{\vec c,\infty})|-2 z_{\vec c,\infty}.$ 
 \end{itemize}

  \begin{thm}\label{step}
If $C$ is a complex scheme in $\rp$, then
$\sig_{b/p}(C)$ viewed as a function of the numbers $b/p$ can be extended to a step function $\sig_{x}(C)$ for $x \in (0,1/2)$ on the interval $(0,1/2)$ whose only discontinuities  
are rational numbers with 
denominators which divide some non-zero entry of $\vec c$. The function $\eta_p$ restricted to the set of primes coprime to the non-zero entries of $\vec c$  has constant value $\nul(C)$. 
 \end{thm}

To remove the ambiguity in the definition of $\sig_{x}$ at those points with non-prime denominators which divide a non-zero entry of $\vec c$, we define $\sig_{x}$ to be the average of the one sided limits
at these points. We assign no particular meaning to these values, but do this simply to avoid ambiguity.
These  averages of one sided limits will also be integers by  (\ref{snc}).

For our example 
$J \amalg 1^- \left<2^-\right> \amalg 2^+ $, the following list describes the signature and nullity functions. The arrows point to either a pair of integers or a single integer. The first integer is   $\sig_{b/p}$ for any point $b/p$ in the interval (or singleton). The second integer is the value $\eta_p$ if there is a number $b/p$ in the interval (or singleton). Thus the second integer assigned to any proper interval is $\nul.$ There is no second integer for  a singleton not of the form $b/p$.
\begin{lstlisting}
(0/1, 1/14) -->  (-1, 0), 1/14 -->  (-2)
(1/14, 1/10) -->  (-3, 0), 1/10 -->  (-5)
(1/10, 1/7) -->  (-7, 0), 1/7 -->  (-8, 1)
(1/7, 1/6) -->  (-9, 0), 1/6 -->  (-8)
(1/6, 1/5) -->  (-7, 0), 1/5 -->  (-7, 0)
(1/5, 3/14) -->  (-7, 0), 3/14 -->  (-8)
(3/14, 2/7) -->  (-9, 0), 2/7 -->  (-10, 1)
(2/7, 3/10) -->  (-11, 0), 3/10 -->  (-9)
(3/10, 1/3) -->  (-7, 0), 1/3 -->  (-8, 1)
(1/3, 5/14) -->  (-9, 0), 5/14 -->  (-8)
(5/14, 2/5) -->  (-7, 0), 2/5 -->  (-7, 0)
(2/5, 3/7) -->  (-7, 0), 3/7 -->  (-6, 1)
(3/7, 1/2) -->  (-5, 0)
\end{lstlisting}

\subsubsection{Formulas for $\Delta$ and the entries of $\vec c$ }

These formulas are useful for  hand and also  computer calculation of $\sig_{b/p}(C)$ and $\eta_p(C)$.

\begin{prop}\label{RMlk} If $C$ has odd type, $\Delta  = -4(  l+2 (\Pi^- -\Pi^+) + (\Lambda^- -\Lambda^+))$.
If $C$ has even type, $\Delta  = -4(  l+2 (\Pi^- -\Pi^+) )$.
\end{prop}

We remark that then the Rohlin-Mishachev restriction can be stated: If $C$ is a dividing real algebraic curve with its complex orientation, then \[-\Delta= \begin{cases} m^2-1 &\text{if $m$ is odd } \\
m^2 &\text{if $m$ is even. }
\end{cases}\]

\begin{prop} \label{codd} If $C$ has odd type,

 \begin{enumerate} 

\item $\vec c_{u_1}=-2-4(\Lambda^--\Lambda^+)$

\item $\vec c_{u_2}=2(\Lambda^- -\Lambda^+)$
\item 
$\vec c_{u_3}=1+2(\Lambda^- -\Lambda^+)$, 

\item For each region $R$, $\vec {c}_{R}=(-1)^{\pari(R)} (1 +2 (\Lambda^-_R -\Lambda^+_R))$

\item For each oval $o$, $\vec c_{o}= \vec s_o(-2 \ep(o) +4 +4 (\Pi^-_o -\Pi^+_o)),$
\end{enumerate}
\end{prop} 

Note using Theorem \ref{form}, a curve of odd type tends to have large $\eta_p$ if there is a region $R$ with $1 +2 (\Lambda^-_R -\Lambda^+_R) =0 \pmod{p}$ and  $R$ is a disk with many holes. Also  it not hard to see that $\nul(C)=0$ for a curve of odd type.

\begin{prop} \label{ceven} If $C$ has even type,

 \begin{enumerate} 

\item $\vec c_{u_1}=-4(\Lambda^--\Lambda^+)$

\item $\vec c_{u_2}= 2(\Lambda^- -\Lambda^+)$
\item 
$\vec c_{u_3}= 2(\Lambda^- -\Lambda^+)$, 

\item For each region $R$, $\vec {c}_{R}=(-1)^{\pari(R)} (2 (\Lambda^-_R -\Lambda^+_R))$

\item For each oval $o$, $\vec c_{0}= \vec s_o( 4 +4 (\Pi^-_o -\Pi^+_o)).$
\end{enumerate}

\end{prop} 

In the same way,  a curve of even type tends to have large $\eta_p$ if there is a region $R$ with $ \Lambda^-_R -\Lambda^+_R =0 \pmod{p}$ and  $R$ is a disk with many holes.

\subsubsection{The graph $\fZ_{\vec c, p}$ is generally rather small}

\begin{prop} \label{scarceodd} If $C$ has odd type, $\fZ_{\vec c, p}$ has no edges and each vertex of $\fZ_{\vec c, p}$ is either indexed by a region or is possibly $u_3$.
\end{prop}

\begin{prop} \label{scarceeven} Suppose $C$ has even type.  

\begin{enumerate}
\item
If $\Lambda^- \ne \Lambda^+   \pmod{p}$, $\fZ_{\vec c, p}$ has no edges and each vertex of 
$\fZ_{\vec c, p}$ is  indexed by a non-outer region.  

\item 
Assume $\Lambda^-  = \Lambda^+   \pmod{p}$. Then $\fZ_{\vec c, p}$  contains  $u_1$, $u_2$, $u_3$ and 
the two edges joining  these vertices.  Let  $R_1$ denote the  outer region.  $\fZ_{\vec c, p}$ contains $R_1$ and the edge joining $R_1$ to $u_1$ if and only if: for every outer oval $o$,   $\Pi^+_o -\Pi^-_o =1 \pmod{p}.$        
$\fZ_{\vec c, p}$ contains no other edges. The only other vertices possibly contained in $\fZ_{\vec c, p}$  are indexed by non-outer regions.
 \end{enumerate}\end{prop}

\subsection{Some hand calculations}\label{hand}

Using  Theorem \ref{form}, and the above formulas, one can derive the following formulas.
We assume that  $\al + \beta \ne 0$.

\[\sig_{1/3}  ( J \amalg 1^- \left<  \alpha^-  \amalg  \beta^+ \right> ) =  
\begin{cases} 
\frac 8 3 (\be-\al)-2 &\text{if $\al=\be\pmod{3}$} \\
 \frac 8 3 (\be-\al+1)-4 &\text{if $\al=\be +1 \pmod{3}$}\\
 \frac 8 3 (\be-\al-1) &\text{if $\al=\be-1\pmod{3}$}
\end{cases}\]

\[\eta_{3} ( J \amalg 1^- \left<  \alpha^-  \amalg  \beta^+ \right>)= 
\al + \be -1 \]

\[\sig_{1/3} ( J \amalg 1^+ \left<  \alpha^-  \amalg  \beta^+ \right> ) =  
\begin{cases} 
\frac 8 3 (\al-\be) + \be -3 \al +1  &\text{if $\al=\be\pmod{3}$} \\

\frac 8 3 (\al-\be-1)+ \be -3 \al +3  &\text{if $\al=\be+1\pmod{3}$}\\

\frac 8 3 (\al-\be+1)+ \be -3 \al -2    &\text{if $\al=\be -1 \pmod{3}$}
\end{cases}\]

\[\eta_{3} ( J \amalg 1^+ \left<  \alpha^-  \amalg  \beta^+ \right> ) = 
\begin{cases} 0 &\text{if $\al \ne \be-1\pmod{3}$}\\
1 &\text{if $\al=\be-1\pmod{3}$}
                         \end{cases}\]

\[\sig_{1/3}  ( J \amalg 1^+ \left<1^+ \left<  \alpha^-  \amalg \beta^+ \right>  \right>) =  
\begin{cases} 
\frac 8 3 (\al-\be)-2 &\text{if $\al=\be\pmod{3}$ } \\
 \frac 8 3 (\al-\be-1) +1 &\text{if $\al=\be +1 \pmod{3}$}\\
 \frac 8 3 (\al-\be+1)-5 &\text{if $\al=\be-1\pmod{3}$}
\end{cases}\]

\[\eta_{3} (  J \amalg 1^+ \left<1^+ \left<  \alpha^-  \amalg \beta^+ \right>  \right>)=
\begin{cases}  
\al + \be  &\text{$\al = \beta  \pmod{3}$} \\
\al + \be -1  &\text{$\al \ne \beta \pmod{3}$}. 
\end{cases}\]

\subsection{Prohibiting two infinite sequences of complex schemes }\label{inf}

\begin{thm} The complex scheme $J \amalg 1^-\left<  12^-  \amalg 15^+  \right> $ is the first of the series of complex M-schemes of degree $2k+1$ where $k=1 \pmod{3}$ (and $k\ge 4$) given by 
$J \amalg 1^- \left<  \alpha^-  \amalg  \beta^+ \right>  $ where $\alpha=(5 k^2-k-4)/6$ and  $\beta=(7 k^2-5k-2)/6$. These complex schemes 
satisfy the Rohlin-Mishachev formula for  complex orientations but  cannot be realized as real algebraic curves.
\end{thm}

They are all prohibited by Theorem  \ref{bound} using the  formulas for $\sig_{1/3}$, and $\eta_{3}$ in subsection \ref{hand}.
We only need the case $\al=\be\pmod{3}$. 
These are also prohibited by earlier unpublished work of Kharlamov.  
The scheme
$J \amalg 1^-\left<  12^-  \amalg 15^+  \right> $ is also prohibited for a curve of degree 9 recently by 
S. Fiedler-Le Touz\'e \cite{F}.

Here is another infinite sequence that we can prohibit in the same way.

   \begin{thm} The degree $2k+1$ complex  M-schemes, $J \amalg 1^+<1^+< \alpha^- \amalg \beta^+>>$ 
for $\alpha =  \frac{1}{6} \left(7 k^2-5k-6\right)$, and $\beta= \frac{1}{6} (5 k^2-k-6)$ where
$k \ge 5$ and $k \ne 1 \pmod{3}$  satisfy the Rohlin-Mishachev formula for  complex orientations but  cannot be realized as real algebraic curves.
\end{thm}
.

\subsection{Some conventions and notations}
 All manifolds are smooth and oriented \footnote{except  for $\rp$ and unoriented spanning surfaces mentioned  in Appendix \ref{even1/4}}. By a plumbing, we mean a 4-manifold obtained by plumbing 2-disk bundles over 2-spheres according  to a tree whose vertices are 
weighted by integers.  By a  graph manifold, we mean the boundary of a plumbing. By a graph link, we mean a link given by specifying a certain number of oriented circle fibers over each 2-sphere. Because we insist that the plumbing graph be a tree (or  disjoint union of trees), our notion of graph manifold and graph link is more restrictive than some usage.
The letter $p$ will always denote a prime number. We use  $\CQ$ to denote a $\BZ_p$-homology sphere. We use $Q$ denote the projective tangent bundle of $\rp$ which is a $\BZ_p$-homology sphere for every odd prime $p$, as $H_1(Q)= \BZ_2 \oplus   \BZ_2$. Sometimes we repeat these conventions or hypotheses for emphasis.

\subsection{Organization}

This introduction began with a summary of our new results concerning real algebraic  curves. 
The proofs are mainly delayed to later sections. 

In \S\ref{plumb.sec1}, we prove  Theorem \ref{pplumb}. 
The  definitions and properties of the signature and nullities of  links in  a $\BZ_p$-homology sphere are reviewed in section \S\ref{section.sig}. These invariants for $L(C) \subset Q$ can interpreted as invariants of a complex scheme $C$. See Definition \ref{defsig}. Section 3 includes  a proof of  Theorem \ref{bound}. Section 3 concludes with Theorem \ref{step2}, which implies Proposition \ref{signulprop} and also
illustrates the benefits of the reparameterization of the signatures given in Definition \ref{reparam}. 
Casson-Gordon invariants are signature invariants of finite cyclic covers of 3-manifolds that are closely related to the signatures of links.
In \S\ref{section.graph}, we  derive formulas for Casson-Gordon invariants and nullity invariants for  graph manifolds. 
In \S\ref{section.glform}, we use the formulas of \S3 to derive  formulas for signatures and nullities of graph links in a 
$\BZ_p$-homology sphere.
In \S\ref{prove}, we prove Theorem \ref{form}, and Propositions \ref{RMlk}, \ref{codd},   \ref{ceven}, \ref{scarceodd}, and  \ref{scarceeven}. 
In \S\ref{section.step}, {we prove Theorem \ref{step}.
We have  three appendices. Appendix \ref{applk}
discusses some formulas for {linking numbers} that we use.
Appendix \ref{sd} has a {Mathematica program that computes $\sig_{b/p}(C)$ and $\eta_{p}(C)$. 
In Appendix \ref{even1/4}, we discuss the {behavior  of $\sig$ and $\nul$ at $1/4$ of complex schemes of even type.}

\section{Plumbing; Proof of Theorem \ref{pplumb}}\label{plumb.sec1}

 We first strengthen the statement of the theorem. This allows 
 a proof by induction on the number of ovals.    In the inductive step, we will delete an empty oval ( i.e. an oval with no ovals in the disk that it bounds). 
 
To describe the strengthening, we must first recall the definition of $L(C).$  To each point $x$ of $C$ we have the tangent line $l_x$ to $C$ at $p$. The pair $(x,l_x)$ is a point in the projective tangent bundle of $\rp$. 
One defines $L(C)= \{  (x,l_x)| x \in C\}$. It is a link in $Q$ with one component lying over each component of $C$ in $\rp$.  We orient $L(C)$ so that the projection to $C$ is orientation preserving on each component. 

We next fix an orientation on $\rp \setminus J$. In each region $R_i$ of $\rp \setminus (C \cup J )$, pick a point $y_i$ , and let $f_i$ denote the fiber in $Q$ over $y_i$, and  assign $f_i$ the orientation given by a counterclockwise rotation of lines (with respect to the orientation chosen for $\rp \setminus J$ )  in $\rp$ through $y_i$. Let $\hat L(C)$ denote the link $L(C) \cup_i f_i$.

 Let $\hat \Gamma(C)$ denote the decorated plumbing graph $\Gamma(C)$ 
with  $\beta_0(\rp \setminus C)$ additional signed arrows: one arrow is attached to each vertex that is associated to a region of $\rp \setminus J$. The sign on each arrow  indexed by $R$ is
given by  $(-1)^{\pari(R)}$. In Figure \ref{example3}, we illustrate $\hat \Gamma(C)$ for $C$ of Figure \ref{example}. We let $\CL(\hat \Ga(C))$ in  $Q$ denote the link obtained by adjoining  to  $\CL(\Ga(C))$ the $\beta_0(\rp \setminus C)$ fibers specified by the added signed arrows. 

 \begin{figure}[h]
\includegraphics[width=2in]{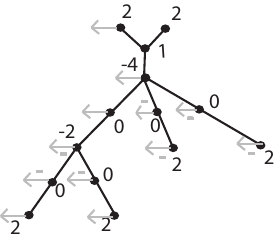}
\caption{The associated decorated graph $\hat \Ga(C)$ for the example complex scheme given on the left of Figure \ref{example}.
} \label{example3}
\end{figure}

Here is the strengthened version of Theorem \ref{pplumb}

 \begin{thm}\label{pplumb+} $\hat L(C)$ and $\CL(\hat \Ga( C))$ both describe  the same  link in $Q$. Moreover, the signature of the  plumbing matrix  $\Ga(C)$ is two. Two fibers to any spherical base in the plumbing, which is indexed by a region of  $\rp \setminus C$, will have linking number zero with each other in $Q$. 
 Two fibers to any spherical base in the plumbing,  indexed by an oval of  $C$, will have linking number two with each other in $Q$. 
\end{thm}
 
 \begin{proof}
 
 In \cite[pages 59-60]{G2}, the first author described how to draw a picture of $L(C)$ inside a description of
 $Q$ as the result of surgery on a $2$-component framed link in $S^3$. $L(J)$ is a one-component link which we denote
 by $\ell.$ Moreover $L(C \setminus J)$ is obtained by a iterated satellite construction starting with a single component  we denote by $f$ (for fiber of the  projective tangent  circle bundle). This is illustrated in Figure 
 \ref{slide} where we describe how to slide $\ell$ over one of the 2-handles in the boundary of 4-dimensional handlebody specified by a framed link.  Figure \ref{quat} shows that the result which can be converted to a plumbing diagram after  two positive blow-ups.
 
  \begin{figure}[h] 
\includegraphics[width=2in]{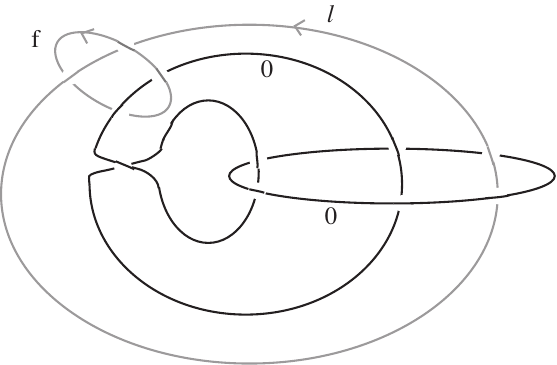} \quad \includegraphics[width=2in]{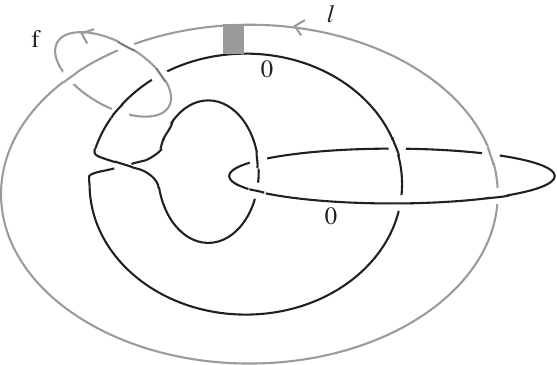}
\caption{On the left is a 2-component framed link description for $Q$ (these components are zero  framed and so  have $0$ written near them. Also shown is  2-component link in $Q$  consisting of $f$ and $\ell$.
On the right the grey box shows the path of the isotopy of $\ell$ in $Q$ to be performed before sliding over one of the handles.
} \label{slide}
\end{figure}
  \begin{figure}[h] 
\includegraphics[width=1.4in]{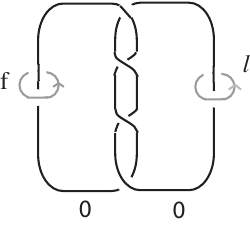} \ \ \includegraphics[width=2in]{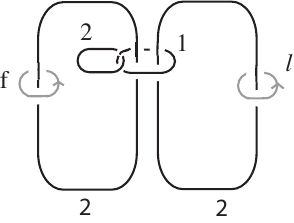}   \includegraphics[width=1in]{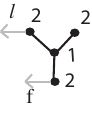}
\caption{On the right is the result of the  handle slide proposed in Figure \ref{slide}, after a further isotopy of the framed link and the two components $\ell$ and $f$. If one blows down in sequence two unknots with framing $1$ in the middle diagram, one recovers the diagram on the left.
The graph link diagram on the right denotes the same link in $Q$.} \label{quat}\end{figure}

It is enough to prove the theorem in the case $C$ is a complex scheme of odd type.
Suppose $C$ has no ovals, then $\hat L(C)$ and $\CL(\hat \Ga( C))$ are both described simply by the plumbing diagram on the left of Figure \ref{quat}. We note that the signature of the plumbing graph is two as when we blow down two plus one framed unknots we get a framed link with linking matrix $\begin{bmatrix} 0 &2\\ 2&0 \end{bmatrix}$.  The fibers over two distinct points  in any one of  2-spheres  that are weighted two have linking number zero. This follows from Proposition \ref{Plumblinking} applied to  the the surgery description on the left of figure \ref{slide}.  
Thus Theorem \ref{pplumb+} holds, if $C$ contains no ovals. 

Now we are set to prove the theorem by induction on the number of ovals. We consider a complex scheme $C$ with some ovals, and suppose the theorem holds for $C'=C$ with an empty  oval deleted.

The link   $\hat L(C)$ is obtained from $\hat L(C')$ by adding two components as follows: take the component of $\hat L(C')$ which is the fiber, say $f'$, over some point  in the region where the  new oval $C \setminus C'$ is to be born and push off an parallel copy of this fiber, say $f$ such that the linking number of $f$  and $f'$ is zero and then adjoin a $(2,1)$ cable of $f'$.  
We want to see that the corresponding change from $\CL(\hat \Ga( C))$  to $\CL(\hat \Ga( C'))$ amounts to the same thing. For this, we draw framed link pictures. First,  we simply consider how  surgery descriptions that correspond to the plumbing diagrams change  in going from $\CL(\hat \Ga( C'))$ to $\CL(\hat \Ga( C))$. 

  \begin{figure}[h]
\includegraphics[width=3.5in]{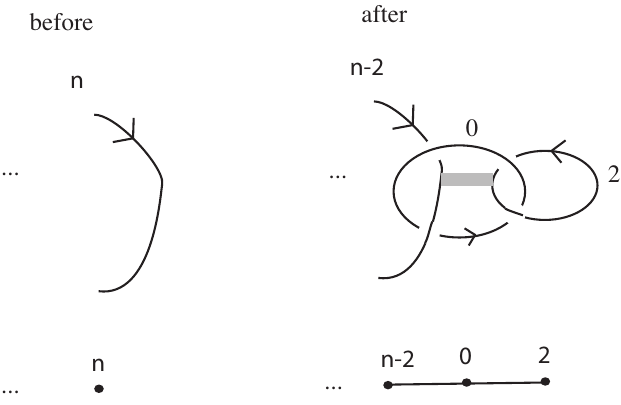} 
\caption{Framed link descriptions and corresponding plumbing diagrams. The grey rectangle indicates a handle slide.
} \label{change}
\end{figure}

Consider Figure \ref{change}.
If we start with the after picture, we can slide the handle with framing $n-2$ over the handle with framing $2$
using the track suggested by the gray rectangle. The curve previously with framing $n-2$ now has framing $n$.
We also have two other components  but one of these with framing  zero is a meridian to the other component.
One may use this zero framed component to unlink everything else from this two component Hopf link which
can then be erased completely as it is a surgery description of $S^3$. Notice that this shows that the signature of the plumbing matrix does not change, as handle sliding corresponds to basis change and the signature of 
$\begin{bmatrix} ? &1\\ 1&0 \end{bmatrix}$ is zero.  Also a positive  meridian  to the  curve framed $n$ in the before picture is isotopic to a negative meridian to the curve framed $2$ in the after picture. 
The assignment of parities to regions that is used in the algorithm to construct $\Gamma(C)$ compensates for this alternation. 
In addition, a pair of  meridians  to the  curve framed 
$n$ in the before  picture, which by induction must have linking number zero, is isotopic to a  pair of meridians to the curve framed $2$ in the after picture. So a pair of meridians to the curve framed $2$ in the after picture also have linking number zero. Because of this a pair of meridians in the plumbing picture correspond to a pair of fibers in the projective tangent circle bundle over the region inside the  new oval.

Finally one  may see that  a  negative meridian to the zero framed component may be slid over the handle corresponding to the 2-framed component so that it is seen to be isotopic to a (2,1) cable around a positive meridian for the 2-framed component. We illustrate this in Figure \ref{ill},  in the case that the oval we add in going from $C'$ to $C$ has odd parity. This is exactly what is needed to see  $\hat L(C)= \CL(\hat \Ga( C))$ using the description of $L(C)$ in \cite{G2}. If we slide two negative meridians to the zero framed component  over the handle corresponding to the 2-framed component in sequence, the result can be seen to be  a (4,2) cable around a positive meridian for the 2-framed component.  This  is the link with two components lying above two nested ovals oriented in the same direction (and so making up a negative injective pair) as in \cite{G2} and so this link has linking number $2$.

The case that the added oval has even parity is similar. 
\end{proof}

\begin{figure}[h]
\includegraphics[width=1.3in]{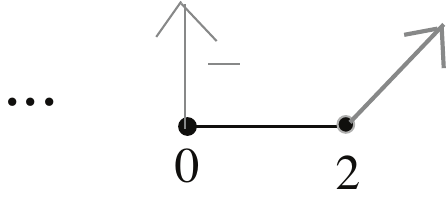} \   \ \includegraphics[width=1.3in]{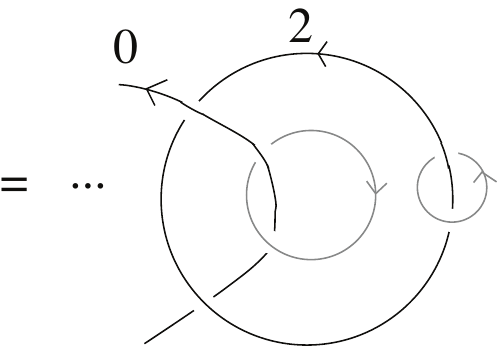} \   \ \includegraphics[width=1.3in]{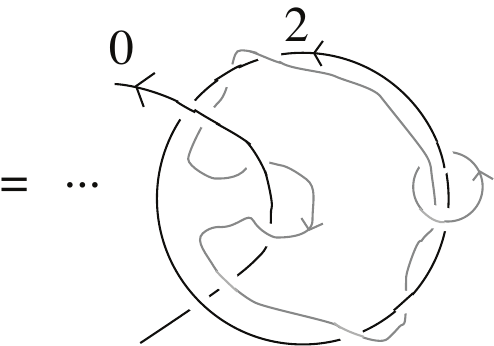}
\caption{} \label{ill}
\end{figure}

\section{Signature and nullity}\label{section.sig}

In this section, we define  signatures and nullities for links in a $\BZ_p$-homology sphere $\CQ$ that are associated to  $\BZ_p$-regular cover of the link complement. We only consider here covers which act in the same way on the cover of each meridian of the link component. For the applications given in \S\ref{sec.intro}, we will study these invariants for  graph links in a graph manifold descriptions of $Q$. Here we may use all odd primes $p$ as, of course $Q$ is a $\BZ_p$-homology sphere for any odd $p$. In \S \ref{section.glform}, we develop formulas for these signatures and nullities of graph links in a general graph manifold which is  a $\BZ_p$-homology sphere. We do this as we think these more general results are interesting in themselves.

There is a general discussion of signatures and nullities for links in rational homology spheres in \cite{G2, G3}.   In \cite{G2,G3}, the author considers more general situations than we need to consider here.  So we sketch the definitions as we use  slight variations of the definitions used previously. The reader should look to the above references for arguments which are omitted.

\subsection{Conventions for $\BZ_p$-covers}\label{cover.sub}

If we have a space $X$ equipped with a cohomology class $\phi \in H^1(X, \BZ_p)$, then $\phi$ describes a regular covering space 
which we will denote $X_\phi$ with $\BZ_p$ as group of covering transformations. 
There is a generator $T_\phi$ for the group of covering transformations such that  ${T_\phi}^{\phi[\ga]}$  sends the initial point of the lift of a loop 
$\ga$ to the terminal point of the lift of $\ga$.  

Similarly if $X$ is the complement of a codimension-2 submanifold $Y$ of a manifold 
$Z$ and $\phi$ assigns to  all the meridians of  $Y$ a non-zero value, we may complete
$X_\phi$ uniquely to a branched $\BZ_p$- cyclic cover of $Z$ branched along $Y$, which we will also denote by $Z_\phi$. Note in this case $\phi$ is a cohomology class of $X$ but not of $Z$.
If $S \subset X$ or $S \subset Z$ , we let $S_\phi$ denote the inverse image in $X_\phi$ or $Z_\phi$.  
Then  the group $\BZ_p$, generated by  the covering transformation $T_\phi$,  acts on $S_\phi$  with orbit space $S$.

If $0<a< p-1$, then ${T_{a \phi}}= {T_\phi}^a$. Thus  the $\zeta_p$ eigenspace for $T_{a \phi}$ is  also the ${\zeta_p}^{\hat a}$ eigenspace for $T_{ \phi},$ where $a \hat a = 1 \pmod{p}.$
Thus \cite[Lemma 7.2]{TW} there are isomorphisms  induced by Galois automorphisms of $\BQ(\zeta_p)$ between  $\CE_i(S_\phi)$  and $\CE_i(S_{a \phi})$. 
Moreover, if we change $\phi$ to $-\phi$,  this induced automorphism is complex conjugation. 
Let $\CE_j(S_\phi)$ denote the $\zeta_p= e^{2 \pi i/p}$-eigenspace for the action of $T$ on $H_j(S_\phi, \BC)$. 
When $S_\phi$ is a  $\BZ_p$ cyclic cover of $S$, $\CE_j(S_\phi)$ is the homology of $S$ twisted by $\phi$.

If $W_\phi$ is the  branched   $\BZ_p$ cyclic cover  of an  4-manifold  (possibly with boundary) $W$ branched along a proper surface $Y$ (possibly empty), then there is a Hermitian intersection  form on $\CE_2(W_\phi)$.
Let $\SignE(W_\phi)$ denote the signature of the Hermitian intersection form on $\CE_2(W)$.
{\it We usually omit the subscript $\phi$, and write $T$ for $T_\phi$, $\CE_2(W)$ for $\CE_2(W_\phi)$, and  $\SignE(W)$ for  $\SignE(W_\phi)$., etc.} We have that $\SignE(W_{-\phi})=\SignE(W_\phi).$

\subsection{Definitions  of signature and nullity }\label{def.sub}

\begin{deff} 
We let  $\phi_{L,p}:H_1(\CQ \setminus L) \rightarrow \BZ_p$ denote  unique homomorphism which takes the value of $1$ on the positive meridians of $L$.
\begin{equation}
\eta_p(L) = {\dim}\left( \CE_1 (\CQ_{\phi_{L,p}}) \right). \notag \end{equation}
\end{deff} 
Using the above remarks about Galois automorphisms,  ${\dim}\left(\CE(\CQ_{a \phi_{L,p}})\right)$  for $a$ prime to $p$ is independent of $a$.

\begin{deff} \label{sigdef} Suppose $W$ is a 4-manifold with $\partial  W=\CQ$  and $G$ is a proper  surface in $W$ with boundary $L\subset \CQ$ with no closed components, and there is an extension of $\phi_{L}$ to $H_1(W \setminus G)$ which we denote simply by $\phi$.  Let $G'$ be an arbitrary transverse push-off  of $G$ and let $G \circ G'$ denote a signed count of the double points.  Let $L'= \partial G'$, the push-off of $L$ in $Q$ specified by $G'$. Let   $\Lk(L, L')$  denote the sum of the linking numbers of components of $L$ with the components of $L'$. For $0<a<p$, we define
\begin{equation}
\sigma_{a/p}(L)=  \SignE(W_{a \phi}) -{\Sign(W)} + \frac 2 {p^2} a (p-a)\left( G \circ G'-\Lk(\partial G, \partial G')\right) 
 \end{equation}
\end{deff} 

But such a $(W,G)$ as defined above may not exist.  But, if not,  there is anyway  such a $(W,G)$ whose boundary is $p$ copies of $(\CQ,L)$. Then one modifies the definition of $\sigma_{a/p}(L)$  by taking $1/p$  times  the output of the above formula.
Proofs that these invariants are well-defined can be found in \cite{G2,G3}.
In the case of links in $Q$, this extra complication of taking $p$ copies of $(Q,L)$ is unnecessary. For a direct argument, see the beginning of the proof of Theorem \ref{step2}.
We have that $\sigma_{a/p}(L)=\sigma_{(p-a)/p}(L)$.
If $a$ is not in the range $0<a<p$ and $a$ is not divisible by $p$, we  define 
$\sigma_{a/p}(L)= \sigma_{c/p}(L)$ where $0<a<p$ and $c =a \pmod{p}$.

If $L$ is null-homologous in $\CQ$, then one may calculate these signatures and nullities in the usual way from a Seifert surface and Seifert pairing \cite{O}. To see this push interior of the Seifert surface into the interior of a collar $\CQ \times I$ on the boundary of $Z$. The argument in \cite[Chapter 12]{Ka} given for links in $S^3$ applies to calculate   $\SignE(\CQ \times I)$. 
In this way, one obtains a branched cover of  $\CQ \times I$ which above $\CQ \times {0}$ is the branched cover $\CQ$ along $L$, and above 
 $\CQ \times {1}$ is the disjoint union of $p$ copies of $\CQ$ being permuted cyclically. One then completes this  with copies of  a 4-manifold with boundary $\CQ$ also being permuted cyclically.

Another parameterization of the signatures 
for links in $Q$ 
is  very  useful.

\begin{deff}\label{reparam} For $p$ odd, and a $L$ link in $Q$, define
 \begin{equation} \sig_{b/p}(L) = \sigma_{\frac{2 b}p}(L).  \notag 
\end{equation}
\end{deff}

As $\sigma_{\frac a p}(L)=\sigma_{\frac{p-a}p}(L)$, we have that 
$\sig_{\frac a p}(L)=\sig_{\frac{p-a}p}(L)$. Thus
 we may as well only consider $\sig_{b/p}(L)$ for $0 \le b \le \frac{p-1}2.$
We define invariants  $\sig_{b/p}(C)$ and $\eta_p(C)$ of the introduction, by these same functions applied to $L(C)$ in $Q$.  

\begin{deff}\label{defsig} For $p$ odd, a complex scheme $C$, and $0 \le b \le \frac{p-1}2$, let
 \begin{equation}\sig_{b/p}(C) = sig_{b/p}(L(C)) = \sigma_{\frac{2 b}p}(L(C)) \notag  
\end{equation}
 \begin{equation} \eta_{p}(C) = \eta_{p}(L(C)).  \notag   
 \end{equation}
\end{deff}

\subsection{Proof of Theorem \ref{bound} }\label{section.bound}

As observed in the introduction, if $C$ is the real part of a dividing real algebraic curve of degree $m$ with its complex orientation, 
then $L(C)$ (with its complex orientation) is the boundary of a connected orientable surface  $G$ properly embedded in $D$ with $\beta_1(G)= (m-1)(m-2)/2.$ 
(see the subsection \ref{DG} for the construction of  $D$ and $G$). Using, for instance, a Mayer-Vietoris sequence for $D$ as the union of $D\setminus G$ and a tubular neighborhood of $G$,
one has that
\begin{equation}\label{mvcomp}
\beta_i(D \setminus G,\BZ_p) =
\begin{cases} 1 &\text{if $i=1$, with the homology generated by a meridian } \\
0 &\text{if $i=3$. }
\end{cases}.
\end{equation}

Let $a$ be a nonzero element of $\BZ_p$.  The $\BZ_p$-branched cover of $Q$ along $L(C)$ given by $a\phi_{L}$ (see section \ref{def.sub})  extends to a unique $\BZ_p$-branched cover of $D$ along $G$. Estimates due to Turaev-Viro \cite{TV} (see also \cite{V2},  \cite[Lemma 7.2]{G3} \cite[Prop. 1.4,1.5]{G1}) allow us to see that 

\begin{equation*}
\dim \CE_i(D \setminus G) =
 0  \quad \text{if $i=1$ or $3$. } \end{equation*}

Since the Euler characteristic of the graded vector space $\{\CE_i(D \setminus G)\}_i$ equals  the Euler characteristic of  $D \setminus G$ or $\chi(D) - \chi(G)$ (by \cite[Prop 1,1]{G1} for instance), it follows that  $\dim (\CE_2(D \setminus G))=(m-1)(m-2)/2$.
Using a Mayer-Vietoris sequence to calculate the (invisible) effect of gluing back the branch set 
 (the cover of each meridian is connected and $\CE_i(S^1)=0$ for i$ \ne 0$)
\begin{equation*}
\dim \CE_i(D) =
\begin{cases} 0 &\text{if $i=1$ } \\
 (m-1)(m-2)/2 &\text{if $i=2$ }\\
0 &\text{if $i=3$. }
\end{cases}.
\end{equation*}

Let $\CE\CI$ denote a matrix for the intersection pairing on $\CE_2(D)$.
We have the long exact sequence of the pair: 
\[ 
\begin{CD}
\CE_2(D) @>\CE\CI>> \CE_2(D,Q)@>>> 
\CE_1(Q) @>>>\CE_1(D)=0.\\
\end{CD}\]
It follows that
$|\Sign(\CE\CI)|+\nullity(\CE\CI) \le (m-1)(m-2)/2$,   and that $\eta_p(L(C))= \nullity(\CE\CI)$. Note that $\Sign (D)=0$, and if we choose a push-off
 $G'$ of $G$ which misses $G$ entirely, then $\Lk (\partial G, \partial G')=0$. Thus, by definition,
  $\sigma_{a/p}(L(C))=  \Sign(\CE\CI)$.
  
  Thus we have $ |\sigma_{a/p}(L(C))| + \eta_p(L(C)) \le (m-1)(m-2)/2.$
Letting $a=2b$, we obtain the sought for formula.  \qed

\begin{thm}\label{step2}
If $L$ is an  link  in $Q$, then
$\sig_{b/p}(L)$ as a function on the set of numbers  $b/p$,  where $p$ is odd, and  $1 \le b \le \frac {p-1} 2$,  can be extended to  a step function with only finitely many discontinuities on the interval  $(0,1/2)$.
Except for finitely many odd $p$,  $\eta_{p}(L)$ is a constant  function of $p.$ 
One also has that 
\begin{equation} \sig_{b/p}(L) +\eta_{p}(L) = \beta_0(L)-1 \pmod{2} \label{snc2} 
\end{equation}
\begin{equation}0 \le \eta_{p}(L) \le \beta_0(L)-1.   \label{nb2}
\end{equation}

\end{thm}

\begin{proof}
We note that there are symmetries of $Q$ which induce any permutation of the three nonzero elements of $H_1(Q)$.
Also the kernel of the map induced by the inclusion  $H_1(Q) \rightarrow H_1(D)$ is generated by the class of the fiber. It follows that, given any link $L \subset Q$, 
we can find a diffeomorphism of $Q$, so that the image of $L$ under this diffeomorphism bounds a  proper surface $G$ in $D$. As diffeomorphisms leave our invariants 
unchanged,  we may assume without loss of generality  that $L$  bounds a proper connected surface in  $D$.

We  will let $\nu$ stand for  tubular neighborhood.
 Let $G'$ be a push-off of $G$ with no multiple points, and let $\hat G= G \sqcup G'$ with boundary $\hat L= L \sqcup L'$.  By definition, $\si_{2b/p}(L)$ is $\SignE(D)$ where $D \setminus \nu(G)$ is equipped with the extension of   $2b \phi_{L}$.  Here  we make  use of the fact that $D$ has a trivial intersection pairing and also that the linking number of $L$ and $L'$ will be zero (by Proposition \ref{linking}). 
 We can identify  $\nu( G)\setminus G$ with $G \times (D^2 \setminus \{0\})$ in such a way that the restriction of $\phi_{L}$ is trivial on $H_1(G).$ As $\phi$ is non-trivial on the meridians of $G$ and $\hat G$,  $\CE_*$ vanishes on  the sphere bundle of $\nu(G),$ 
 and the sphere bundle of $\nu(\hat G).$ 
 In this way using  Mayer-Vietoris, we have  that $\SignE(D)=\SignE(D\setminus \nu(G))$, 
and  $\si_{b/p}(\hat L))$ is $\SignE(D\setminus \nu(\hat G))$ where $D\setminus \nu(\hat G)$ is equipped with the extension of $b \phi_{\hat L,p}$. 
Note we can obtain $D\setminus \nu(\hat G)$ (equipped the extension of $b \phi_{\hat L,p}$) from  $D\setminus \nu(G)$ (equipped the extension of $2b \phi_{\hat L,p}$) by gluing on $G \times D_2$, where $D_2$ is a disk with two holes. Here  $G \times D_2$ is equipped with a cohomology class that is trivial on $H_1(G)$ and  sends the outer boundary component of $D_2$ to $2b$ and the two inner boundary components to $b$, in the obvious way.   
 The  ordinary intersection form on $H_1(G)$ can be given by  a block matrix of the form $$\begin{bmatrix}
0& I_n & 0\\
-I_n& 0&0\\
0&0&0_m\\
\end{bmatrix}.$$ One has $\CE_*(\partial D_2)=0$. By an Euler characteristic argument, $\CE_*( D_2)$   is concentrated in dimension one, where it has dimension one. It follows that  the skew Hermitian intersection form on  
$\CE_1( D_2)$
 is given by a $1\times 1$ matrix with a non-zero
purely imaginary entry. By the Kunneth Theorem, $\SignE(G \times D_2)=0$.  A Mayer-Vietoris then shows that
$\SignE(D\setminus \nu(\hat G))= \SignE(D\setminus \nu( G))$, and thus:

\begin{equation} \si_{2b/p}( L)=\si_{b/p}(\hat L). \label{sid} \end{equation}

Since $\hat L$ is null-homologous, $\si_{2b/p}( L)$ can be computed  by taking the signature of $(1-{\zeta_p}^{b}) V+  (1-{\zeta_p}^{-b})V^t$
where $V$ is a Seifert matrix for a Seifert surface for  $\hat L$. But $\Sign((1-e^{2 s\pi i} )V+ (1- e^{-2 s\pi i}) V^t)$ is a step function of  $s\in (0,1/2)$ with only finitely many jumps, and is the claimed extension.

One also has that  $\eta_p(\hat L)= \nullity((1-e^{2 b\pi i/p} )V+ (1- e^{-2 b\pi i/p}) V^t)$. It is not hard to see that  
$\nullity((1-e^{2 s\pi i} )V+ (1- e^{-2 s\pi i}) V^t) = \nullity(V -  e^{-2 s\pi i} V^t)$  is a constant function of $s$ except at finitely many points $s\in (0,1/2)$. 
This nullity is the nullity of $V -  x V^t$ over the field of  complex rational functions (i.e. the quotients of two polynomials with complex coefficients).

As $\dim( \CE_1(D_2))=1$, and $\dim( \CE_0(D_2))=0$, we have that $\dim\left( \CE_1(L \times D_2)\right)=\beta_0(L)$. 
As $ \CE_*(L \times \text { a boundary component of $D_2$})$ vanishes,   a Mayer-Vietoris sequence now yields  
\begin{equation} \eta_p(\hat L)= \eta_p(L)+ \beta_0(L). \label{nr} \end{equation}
Thus except for finitely many values of $p$,  $\eta_p(L)$ is  a constant function of $p$. 

From the Seifert surface point of view, it is easy to see that:
\begin{equation} \si_{b/p}(\hat L) +\eta_{p}(\hat L) = \beta_0(\hat L)-1\pmod{2}  .\label{snc3}
\end{equation} 
Using (\ref{snc3}), (\ref{sid}) and (\ref{nr}), one obtains (\ref{snc2}). One has that $\dim(H_1(Q\setminus L, \BZ_p))= \beta_0( L).$ Then the estimates for eigenspace homology used in section \ref{section.bound} allow us to conclude (\ref{nb2}). \end{proof}

\begin{rem} {\em The fact that  our signatures,  parameterized in this way, are the values of a step function was suggested by the work of  Cha and Ko \cite{CK}.  Although we did not use their work directly, their idea of a generalized Seifert surface suggested  $\hat L$.
We remark that their statement  \cite[p.1163]{CK} that all the signatures that the first author defined in \cite{G1} can be derived from their signature jump function does not apply to signatures indexed by a  non-zero $\gamma$  \cite[p.310-311]{G1}. }
\end{rem}

\section{Casson-Gordon invariants of certain graph manifolds}\label{section.graph}

Casson and Gordon \cite {CG} defined  certain invariants $\si(M,\phi)$ (they are disguised forms of the Atiyah-Singer $\al$-invariant) of a closed 3-manifold
equipped with a finite cyclic cover specified by  a character $\phi: H_1(M_\Ga) \rightarrow \BZ_p$ (we only consider the prime case in this paper). They are discussed in \cite{G1,G2} where the author defines a related nullity invariant $\eta(M,\phi)$ which is just the dimension of the $e^{2 \pi i/p}$-eigenspace for the action of the generating covering transformation on the first homology of the cover associated to $\phi$.

If $\Gamma$ is a disjoint union of trees  whose vertices are weighted by integers,  let $W_\Gamma$ denote the 4-manifold given by plumbing disk bundles over  2-spheres according to $\Gamma$. In this section, we consider any tree,  i.e. our graph does not need to be related to a complex scheme in $\rp.$  Let $A_\Ga$ be the matrix associated to the weighted tree $\Ga$, or equivalently the matrix for the intersection form on $H_2(W_\Gamma)$ with respect to the basis of $H_2(W_\Gamma)$ given by the fundamental classes of the $2$-spheres. 

Let $M_\Gamma$ denote the boundary of $W_\Ga$. It is a graph manifold.  We have that $H_1(M_\Ga)$ is generated by the oriented fibers of the circles bundles (on the boundary)  over points  on the 2-sphere bases away from where the plumbing takes place. Relations  among these generators are given by the columns  of $A_\Ga $. These columns form a complete set of relations.  Every connected regular $\BZ_p$- covering space of 
$M_\Ga$ (with a choice of generator of the group of covering transformations) can be specified by  a row vector $\vec \fc$ such that 
$\vec \fc   A_\Ga \equiv 0 \pmod{p}$. Moreover $\vec \fc$ is not congruent to zero modulo $p$. We call  $\vec \fc$ a $p$-characteristic vector. If two such vectors agree modulo $p$, they describe the same cover. We can uniquely specify the cover by insisting that the entries of  $\vec \fc$ lie in the range $[0,p-1]$.  Let 
$\phi_{\vec \fc}: H_1(M_\Ga) \rightarrow \BZ_p$ denote the non-zero homomorphism which classifies the  $\BZ_p$-regular covering space which is specified by the vector  $\vec \fc$. 

We sometimes record this extra information for a $\BZ_p$-covering space  of $M_\Ga$ by labeling the vertices with the corresponding entry of $\vec \fc$ in parenthesis. Then the compatibility condition $\vec \fc_\Ga   A_\Ga \equiv 0 \pmod{p}$ becomes a local condition at each vertex: 
the values in parenthesis times the weight plus the sum of the entries in parentheses of the adjoining vertices should be zero modulo $p$.
We will say $\vec \fc$  is  $p$-characteristic for $\Ga$ at a vertex $\fv$ if this last condition holds, i.e. $(\vec \fc   A_\Ga)_{\fv} \equiv 0 \pmod{p}.$

If $x\in \BZ_p$, let $r_p(x)$ denote the unique integer congruent to $x$ modulo $p$ in the range $[0,p-1]$.  If $\vec x$ is instead a vector in ${\BZ_p}^n$, for some $n$, $r_p(\vec x)$  is defined in this same way but  entry by entry.

\begin{itemize}
\item If $\fv$ is a vertex of $\Ga$, and  $\vec {\fc}$ is  vector indexed by the vertices of  $\Gamma$, we let  
$\vec {\fc}_\fv$ be the coefficient of $\vec {\fc}$ associated
to $\fv$.

\item  Let $Z_{\vec {\fc}}$ be the subgraph of $\Ga$ consisting of vertices $\fv$ such that $r_p(\vec {\fc}_\fv) =0$   and all the edges joining such vertices to each other. Here $Z$ stands for zero.

\item We let $\fZ_{\vec {\fc}}$ be the subgraph of $\Ga$ consisting of vertices $\fv$ such that $r_p(\vec {\fc}_\fv) =0$  that are not connected by an edge to 
any vertex $\fw$ with  $r_p(\vec {\fc}_\fw) \ne 0$,  and all the edges joining such vertices to each other.  Here we can think of $\fZ$ as standing for ``zero and not contiguous to non-zero."

\item Let $e_{\vec {\fc}}$ be the number of edges in $\Ga$  
that join  two vertices, say  $\fv$ and $\fw$, with $r_p(\vec {\fc}_\fv)\ne 0$ and $r_p(\vec {\fc}_\fw)\ne 0$.

\end{itemize}

\begin{thm}\label{cggraph} If $\Gamma$ is a connected tree, 
 \begin{align}
 \sigma(M_\Ga, \phi_{\vec {\fc}}) &= 
 \frac 2 {p^2} \left(r_p(\vec {\fc}) (A_\Ga )r_p(-\vec {\fc})^t \right)+ \Sign( A_{\fZ_{\vec {\fc}}} )-e_{\vec {\fc}} - \Sign(A_\Ga)
 \label{scg}  \\ 
\eta (M_\Ga, \phi_{\vec {\fc}}) 
 &=
  \nullity( A_{{\fZ_{\vec {\fc}}} }) + |v(\fZ_{\vec {\fc}})|-2|  v( Z_{\vec {\fc}} ) |+|v(\Ga)|-e_{\vec {\fc}}-1.
   \label{ncg}
 \end{align}
 \end{thm}

\begin{proof}
This result, in the case that $r_p(\vec {\fc}_\fv) \ne 0$ for all vertices $\fv$ of $\Ga$,  was proved in \cite{G1}.
The formula for $\sigma(M_\Ga, \phi_{\vec {\fc}})$, in this special case,   was reproved in a very different way in \cite{G4}.
We let $\CN_{ \vec {\fc}}$ be the subgraph consisting of vertices $\fv$ such that $r_p(\vec {\fc}_\fv) \ne 0$ and all the edges joining such vertices to each other. Here $\CN$ stands for non-zero. Let $\vec {\fc}_\CN$ be the corresponding subvector of $\vec {\fc}$. Then $\vec {\fc}_\CN$ is $p$-characteristic for $\CN_c$ and is non-zero modulo $p$ at each  vertex. Recall $A_{\CN_c}$ denotes the  matrix associated to ${\CN_c}$.  Then we note that
$r_p(\vec {\fc}) A_\Ga (r_p(-\vec {\fc}))^t= r_p(\vec {\fc}_\CN) A_{\CN_{\vec {\fc}}}  (r_p(-\vec {\fc}_\CN))^t$. Thus according to \cite[(3.7) (3.8)]{G1}

\begin{align}\label{nice}
\si(M_{\CN_c}, \phi_{\vec {\fc}_\CN}) & =
 \frac 2 {p^2} \left(r_p(\vec {\fc}) (A_\Ga) r_p(-\vec {\fc})^t \right) -e_{\vec {\fc}} - \Sign({A_{\CN_{\vec {\fc}}}}) \\ \label{niceeta}
\eta_{p}(M_{\CN_c}, \phi_{\vec {\fc}_\CN})& =0.
 \end{align}

Let $\fJ_{ \vec {\fc}}$ denote the set of edges of $\Ga$ which join vertices of $\CN_{ \vec {\fc}}$ to vertices of ${Z_{ \vec {\fc}}}.$ 
One can construct $W_{\Ga}$ by 
plumbing
 together the two plumbings $W_{\CN_{ \vec {\fc}}}$ and $W_{Z_{ \vec {\fc}}}$  along copies of $D^2 \times D^2$ that are indexed  $\fJ_{ \vec {\fc}}$.  
  We identify each copy of $D^2 \times D^2$  at these plumbing sites, 
 so that
each copy of $D^2 \times \{0\}$   lies on a  2-sphere indexed by  the vertices of ${Z_{ \vec {\fc}}},$ and each copy of  $\{0\} \times D^2$   lies on a 2-sphere indexed by 
the vertices of ${\CN_{ \vec {\fc}}}.$
 Let $W'_{Z_{ \vec {\fc}}}$ be $W_{Z_{ \vec {\fc}}}$ with a copy of  $\Int(D^2) \times D^2$  deleted at each site. So  $W_{\Ga}$ is obtained by gluing together $W_{\CN_{ \vec {\fc}}}$ and $W'_{Z_{ \vec {\fc}}}$  along copies of $S^1 \times D^2$ that are indexed by $\fJ_{ \vec {\fc}}$. We will denote these solid tori by $T_j$ where $j \in \fJ_{ \vec {\fc}}$.  These
 $T_j$ can be identified with subsets  of $W'_{Z_{ \vec {\fc}}}$ and can be identified with subsets  of  
 $M_{\CN_{ \vec {\fc}}} \subset W_{\CN_{ \vec {\fc}}}$. 
 The homomorphism  $\phi_{\vec {\fc}_\CN}$ is non-zero on
 the cores of the $T_j$.

 Within $W_{\Ga}$, there is a cobordism $V$  with boundary $M_\Ga \sqcup -M_{\CN_{\vec {\fc}}}$ which can be constructed   by attaching 
  $W'_{Z_{ \vec {\fc}}}$ to $(-\partial M_{\CN_{\vec {\fc}}}) \times I$ along the union of the solid tori  $T_j$.
             The homomorphism  $\phi_{\vec {\fc}_\CN}$   extends  uniquely to all of $H_1(V)$.
 Thus $V$ and the associated $\BZ_p$-cover of $V$   can be used to compute
 $ \sigma(M_\Ga, \phi_{\vec {\fc}}) -\si(M_{\CN_{\vec {\fc}}}, \phi_{\vec {\fc}_\CN}).$
 One can also obtain  $V$ by deleting $(W_{\CN_{ \vec {\fc}}}\setminus \text{a collar on its boundary})$ from   $W_{\Ga}.$
 Thus  $$\Sign(V)=  \Sign( A_{\Ga} )-\Sign( A_{\CN_{ \vec {\fc}}}).$$

There is a deformation retraction of  $W'_{Z_{ \vec {\fc}}}$ to the 2-complex $X$  given as a union of some 2-spheres  $S_\fv$
and indexed by $\fv \in v(\fZ_{\vec {\fc}})$  and  some punctured   2-spheres\footnote{By a punctured 2-sphere,  we actually mean a 2-sphere with an open disk neighborhoods of some isolated  points removed.} $S_\fw$  indexed by  $ \fw \in v(Z_{\vec {\fc}})\setminus v(\fZ_{\vec {\fc}})$. These spheres and punctured spheres are  identified along points as specified by the edges of the tree  $Z_{ \vec {\fc}}$. 
  The cover of each punctured two sphere is a connected surface with non-empty boundary. The cover of each $S_\fv$ 
  is consists of $p$ copies of $S^2$ being permuted  cyclically.

 As $\CE_*(T_j)=0$,  $\CE_2(V)=\CE_2(\partial (\CN_{\vec {\fc}}) \times I) \oplus \CE_2(W'_{Z_{ \vec {\fc}}})$, and the summand 
 $\CE_2(\partial (\CN_{\vec {\fc}}) \times I)$ is in the radical of the intersection pairing on $\CE_2(V)$.
  Thus 
\begin{equation}\SignE(V)= \SignE(W'_{Z_{ \vec {\fc}}}).\notag\end{equation}
   Recall that  $W'_{Z_{ \vec {\fc}}}$ and its cover  deformation retract to the  2-complex $X$ and its cover, respectively. The  covers of the punctured 2-spheres  have no 2-dimensional homology 
   while the trivial covers of the connected 2-spheres each contribute one generator to
$\CE_2(X)$. We have that $\CE_2(W'_{Z_{ \vec {\fc}}})=  \CE_2(W_{\fZ_{ \vec {\fc}}})$.
Moreover the cover of  $W_{{\fZ_{\vec {\fc}}}}$ is trivial. Thus the Hermitian intersection form on 
$\CE_2(
W
_{\fZ_{\vec {\fc}}}
)$
is 
 given by 
 $A_{\fZ_{\vec {\fc}}}$.  
 Thus \begin{equation} \SignE(V)= \SignE(W'_{Z_{ \vec {\fc}}})=\Sign( A_{\fZ_{\vec {\fc}}} ).\notag\end{equation}
 One may conclude that: 
\begin{equation}\label{diff}  \sigma(M_\Ga, \phi_{\vec {\fc}}) -\si(M_{\CN_{\vec {\fc}}}, \phi_{\vec {\fc}_\CN}) = \SignE(V)-\Sign(V)=
\Sign( A_{\fZ_{\vec {\fc}}} )- \Sign( A_{\Ga} )+\Sign( A_{\CN_c}). 
\end{equation}
Equation  (\ref{scg}) follows from (\ref{nice}) and (\ref{diff}). 

 We have that: \[M_\Ga=\left( M_{\CN_{ \vec {\fc}}} \setminus  \sqcup \Int  T_j \right) \cup_{\cup_j \partial T_j} 
  \left( \partial ( W'_{Z_{\vec {\fc}}}) \setminus  \sqcup_j \Int T_j \right).\]
 By (\ref{niceeta}),  
$\CE_1(M_{\CN_{ \vec {\fc}}})=0$. Since $\CE_1(\partial T_j)=0$, and $\CE_1(T_j)=0$, we may conclude that 
$\CE_1\left( M_{\CN_{ \vec {\fc}}} \setminus  \sqcup_j \Int T_j \right)=0$. Similarly one has that 
$$\CE_1\left(  \partial ( W'_{Z_{\vec {\fc}}}) \setminus  \sqcup_j T_j \right)= \CE_1\left( \partial ( W'_{Z_{\vec {\fc}}})\right).$$
Thus $$\eta (M_\Ga, \phi_{\vec {\fc}})  = \dim \left(\CE_1\left( M_\Ga \right)\right)=  \dim \left(\CE_1\left( \partial ( W'_{Z_{\vec {\fc}}})\right)\right).$$

The  connected cover  of $X$  has eigenspace homology that is concentrated  in dimensions one,  and two. 
Viewing $W'_{Z_{\vec {\fc}}}$ as the union of 
$|v(\fZ_{\vec {\fc}})|$ 2-spheres and  $\fJ_{ \vec {\fc}}$ punctured 2-spheres and using an eigenspace Mayer-Vietoris sequence, we have that  $$\dim \left(\CE_2\left( W'_{Z_{\vec {\fc}}}\right)\right)  = |v(\fZ_{\vec {\fc}})|.
 $$
 
The Euler characteristic of $X$  
 is given by
 $ 2|v(Z_{\vec {\fc}})|-e_0$, where $e_0$ is the number of edges  of $\Ga$,  
 for which for at least one of  its endpoints $\fv$, we have $r_p(\vec {\fc}_\fv) = 0$. We have that  $e_0= |v(\Ga)|-e_{\vec {\fc}}-1$. But the Euler characteristic  of $\CE_*(X)$ is also the Euler characteristic of $X$. It follows that 
$$ \dim(\CE_1(W'_{Z_{\vec {\fc}}}))= \dim \left(\CE_2\left( W'_{Z_{\vec {\fc}}}\right)\right) - \chi(\CE_*(X))=
|v(\fZ_{\vec {\fc}})| -2|  v( Z_{\vec {\fc}} ) |+|v(\Ga)|-e_{\vec {\fc}}-1.$$

We have an exact sequence:
\[ 
\begin{CD}
\CE_2(W'_{Z_{\vec {\fc}}}) @>A_{Z_{\vec {\fc}}}>> \CE_2(W'_{Z_{\vec {\fc}}},\partial (W'_{Z_{\vec {\fc}}}))@>>> 
\CE_1(\partial (W'_{Z_{\vec {\fc}}}) )@>>>\CE_1(W'_{Z_{\vec {\fc}}}) @>>> 0\\
\end{CD}
,\]
where the last term is zero as $ \CE_1(W'_{Z_{\vec {\fc}}},\partial (W'_{Z_{\vec {\fc}}}))= \CE^3(W'_{Z_{\vec {\fc}}})= \CE^3(X)=0$. Here we use $\CE^i$ to denote the eigenspace in the $i$th-cohomology of the cover. 
The exactness of the sequence implies (\ref{ncg}).

\end{proof}

\section{Formulas for signatures of graph links}\label{section.glform}

Let $\CQ$ be a $\BZ_p$-homology sphere, and $L$ be an  link in $\CQ$ which can be described as a graph link by a plumbing graph $\Ga$ which is further decorated by signed arrows.  
We note since the boundary of the plumbing is a $\BZ_p$-homology sphere, the plumbing graph must be a connected tree.  
Let $s$ be the vector indexed by the vertices of $\Ga$ whose $v$th entry is the signed count  of arrows with tail at $v$. 
 This follows from the long exact sequence of the pair $(W_\Gamma, \CQ)$.
We  let $\de=  \vec s {A_\Ga}^{-1} {\vec s\ }^t \in   \BQ $.  We denote by $\vec u$ the  vector indexed by $V(\Ga)$ with the $v$-th entry of $-s (A_\Ga)^{-1}$. The entries of $\vec u$ lie in $\BQ(p)$, the ring of rational numbers whose denominators are not divisible by $p$.
Let $\Gamma^+$ be obtained from $\Gamma$ by converting  arrowheads into vertices weighted zero and  form $\vec {u}^+$ by  assigning $\pm 1$ to these new vertices according to the sign of the arrows. The vector $\vec {u}^+$ will not in general be $p$-characteristic for $\Gamma^+$, but we still define 
${\fZ_{\vec {u}^+}}$, $A_{\fZ_{\vec {u}^+}}$, and $e_{\vec {u}^+}$  as in the previous section.
If  $x= \frac \alpha \beta \in \BQ(p)$ where $\alpha$ and $\beta$ are integers and $\beta$ is not divisible by $p$, we can pick an integer $\beta^*$ such that $\beta \beta^* =1 \pmod{p}$ and define $r_p(x)=r_p( \alpha \beta^*)$.  We extend  $r_p$  to vectors in the same way as before.

\begin{thm} \label{cggl} For $ a  \ne 0 \pmod{p}$, and for $L$ described in the paragraph above, we have:
\begin{align} 
  \sigma_{a/p}(L) =  &\frac 2 {p^2}\left( r_p(a{\vec u}^+) (A_{\Ga^+}) r_p(-a{\vec u}^+)^t
   +r_p(a) r_p(-a) \de \right)- \label{asss} \\ &\Sign(A_{\Gamma}) +\Sign( A_{\fZ_{\vec {u}^+}} )-e_{\vec {u}^+}
   \notag
\\
 \eta_p(L)= &
  \nullity(A_{
  \fZ_{
  {\vec u^+}
  } 
  })
   +
   |v(\fZ_{ {\vec u^+} })|- 2|  v( Z_{{\vec u^+}})  |+|v(\Ga^+)|-e_{{\vec u}^+}-1.
 \label{alme}
\end{align} 
One has that  ${\vec u}^+$ is $p$-characteristic  for $\Ga^+$ at every vertex of $\Gamma.$

\end{thm}

\begin{proof}
For the proof we assume $0 <a <p$, then $r_p(a) r_p(-a)  =( p-a)a$.
Suppose that we have picked $W$, a 4-manifold with $\partial  W=\CQ$,  and $G$, a proper  surface in $W$ with boundary $L\subset \CQ$ with no closed components, and  an extension of $\phi_{L}$ to $W \setminus G$.  So by (\ref{sigdef}),
\begin{equation}
{\SignE(W)} -{\Sign(W)} =    \sigma_{a/p}(L)+\frac 2 {p^2} a (p-a) \left(\Lk(\partial G, \partial G')- G \circ G'\right).  \label{defag}
\end{equation}

Each  graph link description of a graph link $L$ leads to a choice of a push-off $L'$ of $L$. This push-off
consists of a nearby fiber in the graph manifold description. We  pick $G'$ so that $\partial G'=L'$ with respect  to our graph link description. By  Proposition \ref{linking}, one has that 
$$\Lk(L,L')= -\de.$$ 

We can do framed surgery to $\CQ$ along $L$ so that $\phi_{L}$ extends to the result of surgery $\CQ^*$. 
Let $f_i$ for $1 \le i \le \be_0(L)$ denote the framings of this surgery with respect to the push-off $L'$. These $f_i$ are only determined modulo $p$, but fix a choice.
Let $U$ be the result of adding the corresponding 2-handles to $\CQ\times I$. Let $F$ denote the cores of the 2-handles union $L \times I$. We have that $\phi_{L}$ extends uniquely
to $U \setminus F$. We also consider $V= W \cup_\CQ U$ with boundary $\CQ^*$, containing the closed surface $E=G\cup_{L}  F.$
We note that there is a push-off $E'$ of $E$ such that
\[E \circ E'= G \circ G' +f \text{ where } f=\sum_{i=1}^{\be_0(L)} f_i.\] The extension of $\phi_{L}$ to $W\setminus G$ and  $U \setminus F$ gives us an cohomology class which we denote simply by $\phi \in H^1(V \setminus E, \BZ_p)$. We will also use $\phi$ for its various restrictions to subsets. In the rest of the proof,
we wish to discuss the branched cover of $V$ along $E$ classified by $a \phi$. We will omit the use of $a \phi$ as a subscript. As in \cite[Prop 3.5]{G1}, the Casson-Gordon invariant is given by
\begin{equation}
 \si(\CQ^*,a\phi)= \SignE(V) -{\Sign(V)}+2a(p-a)/{p^2}\left(  G \circ G' +f \right). \label{cga}
\end{equation} 
We decompose $U$ into $P$ the product  with $I$ of  the exterior of $L$ in $\CQ$ and $Q=U \setminus P$. The intersection form on $\CE_2(P)$ is  identically zero,  $\CE_*(Q)=0,$  and $\CE_*(P \cap Q)=0.$
It follows that $\CE_2(U)$ is identically zero. 
Thus $\SignE(U)=0$.
Novikov additivity gives us   that
\begin{equation}
  \SignE(V)= \SignE(W) \quad  \text{and} \quad \Sign(V)= \Sign(W)+  \Sign(U). \label{nov}
\end{equation}

 By (\ref{defag}), (\ref{cga}), and  (\ref{nov}),
 we have that:
 \begin{equation*}
 \si(\CQ^*,a\phi)=   \sigma_{a/p}(L)+ \frac 2 {p^2} a (p-a) \left( \Lk(L,L') +f \right)- \Sign(U). 
\end{equation*}

Recall $L$ is a graph link described by plumbing diagram $\Ga$   decorated with signed arrows  representing the components   of $L$. We enlarge $\Gamma$ by converting the arrowheads to vertices weighted by the $f_i$ chosen above. Let this new weighted graph be denoted
$\Gamma^*$.  
We have that $\CQ^*=M_{\Gamma^*}.$
Also by Novikov additivity, 
\[\Sign(U)= \Sign(A_{\Gamma^*})-\Sign(A_{\Gamma}).\]

Combining these results we have: 
 \begin{equation}
 \si(\CQ^*,a\phi)=   \sigma_{a/p}(L)+\frac 2 {p^2}a (p-a) (f- \de )+ \Sign(A_{\Gamma})-\Sign(A_{\Gamma^*}).
\label{sym}
\end{equation}

 Our derivation will proceed by using Theorem \ref{cggraph} to evaluate
$\si(\CQ^*,a\phi)$, substituting this new expression for $\si(\CQ^*,a\phi)$ in (\ref{sym}), then solving the resulting equation  for  $\sigma_{a/p}(L).$  

To evaluate using Theorem \ref{cggraph}, 
we need  the vector ${\vec u}^*$ which encodes the values of $\phi$ on the meridians of the surgery description
of $\CQ^*$ that is  given by $\Ga^*$. By Proposition \ref{Plumblinking},
the vector ${\vec u}$ indexed by $v \in V(\Ga)$ has $v$-th entry given by the value of $\phi$ on a meridian to the  2-sphere indexed by $v$. 
Let ${\vec u}^* $ be the vector indexed by  $v(\Ga^*)$ obtained by adjoining to the vector $\vec u$ some extra  entries, indexed by the vertices which replace the arrowheads of the arrows, which are  $\pm 1$ according to the signs on the arrows. We have that  ${\vec u}^*={\vec u}^+$.   Thus $r_p(a{\vec u}^*)$ encodes the values of $a\phi$ on the meridians of the surgery description of $\CQ^*$, and ${\vec u}^*$ is $p$-characteristic for $\Ga^*$.

As the weights of $\Ga^+$ and  $\Ga^*$ agree on the vertices of $\Ga$, we see 
that  
${\vec u}^+$ is $p$-characteristic  for $\Ga^+$ at every vertex of $\Gamma,$
as claimed in the Theorem.

 By Theorem \ref{cggraph}: 
\begin{equation}
\sigma(\CQ^*, a\phi) = 
 \frac 2 {p^2} \left( r_p(a{\vec u}^*)( A_{\Ga^*}) r_p(-a{\vec u}^*)^t \right)+ \Sign( A_{\fZ_{\vec {u}^*}} )-e_{\vec {u}^*} - \Sign(A_{\Ga^*}). \label{lol}
 \end{equation}

Using (\ref{lol}) to substitute for  $\si(\CQ^*,a\phi)$ in (\ref{sym}) then simplifying, we obtain 
\begin{align} 
  \sigma_{a/p}(L) =  &\frac 2 {p^2}\left( r_p(a{\vec u}^*) (A_{\Ga^*} ) r_p(-a{\vec u}^*)^t
   +a (p-a) ( \de-f )\right)- \label{as} \\ &\Sign(A_{\Gamma}) +\Sign( A_{\fZ_{\vec {u}^*}} )-e_{\vec {u}^*}.
   \notag \end{align}

One may check
 that the right hand side of (\ref{as})  remains the same if we simultaneously replace all the 
$f_i$ by zero (and thus $f$ by zero),  $\Gamma^*$ by $\Ga^+$, ${\vec u}^*$ by  ${\vec u}^+$, and $e_{\vec {u}^*}$ by $e_{\vec {u}^+}.$ So we make this simplifying change to the left hand side and obtain  
(\ref{asss}).

We also have that  
 $\CQ$ and $\CQ^*$ are related by surgery on a link with $\be_0(L)$ components. One may use the eigenspace Mayer-Vietoris sequences coming from this surgery to see that  
$\eta_p(L) = \eta(\CQ^*,\phi)$. Thus we obtain
\begin{equation} \eta_p(L)= 
  \nullity(A_{
  \fZ_{
  {\vec u}^*
  } 
  })
   +
   |v(\fZ_{ {\vec u}^* })|- 2|   Z_{{\vec u}^*}  |+|v(\Ga^*)|-e_{{\vec u}^*}-1.
 \label{alm}
\end{equation}
The values of  ${\vec u}^*$ on the vertices added to $\Gamma$ to form  
both $\Gamma^*$ and  
$\Gamma^+$ are weighted non-zero. So
$ \nullity(A_{
  \fZ_{
  {\vec u}^*
  } 
  })
= \nullity(A_{
  \fZ_{
  {\vec u^+}
  } 
  })
$, $\fZ_{ {\vec u}^* }=\fZ_{ {\vec u^+} }$,
$Z_{{\vec u}^*}=Z_{{\vec u^+}}$ 
and 
$e_{{\vec u}^*}=e_{{\vec u}^+}$.  In this way, we obtain (\ref{alme}) from   (\ref{alm}). 
\end{proof}

\section{Proofs of  Theorem \ref{form}, and Propositions \ref{RMlk}, \ref{codd},   \ref{ceven}, \ref{scarceodd}, and  \ref{scarceeven}}\label{prove}

\begin{proof}[Proof of Theorem \ref{form}]
We apply Theorem \ref{cggl} with $\CQ=Q$, $L =L(C)= \CL(\Ga( C))$ constructed in  section \ref{al}
Then we let $a=2b$, and note $a{\vec u}^+= b {\vec c}^+$, and observe
$2 \de= \Delta$.  We recall that $\Sign(A_{\Gamma(C)})=2$ by Theorem \ref{pplumb+}. Taking into account  Definition  \ref{defsig}, we obtain:

\begin{align*}
 \sig_{b/q}(C)
 &=  \frac 2 {p^2}\left( r_p(b{\vec c}^+)  (A_{\Ga^+}) r_p(-b{\vec c}^+)^t
   +b (p-2b)\Delta \right)+ 
    \Sign( A_{\fZ_{{\vec c}^+}} )-e_{{\vec c}^+}-2  \\
 \eta_p(C)= &  \nullity(A_{
  \fZ_{
  {\vec c}^+
  } 
  })
   +
   |v(\fZ_{ {\vec c}^+ })|- 2|  v( Z_{{\vec c}^+} ) |+|v(\Ga^+)|-e_{{\vec c}^+}-1.   \end{align*}
   
   In the introduction we used notation which emphasized the role of $p$: we have that $\fZ_{ {\vec c}^+ }=\fZ_{\vec c, p}$,
    $e_{\vec c, p}=e_{{\vec c}^+}$, and $z_{\vec c, p}= v( Z_{{\vec c}^+} )$. 
        In this we way,  we obtain the stated formulas for $ \sig_{b/p}(C)$.    As the Euler characteristic of the tree $\Ga^+$ is one, 
$\fE_{\vec c, p} =|v(\Ga^+)|-e_{\vec c, p}-1.$  We use this to rewrite the formula for $\eta_p(C)$.
      \end{proof}
      
      \begin{proof}[Proof of Proposition \ref{RMlk}] Form an auxiliary curve $C'$ by drawing a parallel curve to each oval of $C$ and, if $C$ is of odd type, add a curve which is near $J$ and meets $J$ in one point. Then by Proposition \ref{Plumblinking}, $\Lk(L(C), L(C'))=-\Delta/2.$ On the other hand according to  \cite[ Prop 7.1]{G2},  $\Lk(L(C), L(C')) = 2(  l+2 (\Pi^- -\Pi^+) + (\Lambda^- -\Lambda^+)).$
   \end{proof}
   
   \begin{proof}[Proof of Propositions \ref{codd} and  \ref{ceven}]
   It follows from Proposition \ref{Plumblinking} that the entries in $\vec c$ corresponding to a vertex of $\Gamma$ is  twice the linking numbers of the fiber
 over the 2-sphere indexed by that vertex with $L(C)$ in $Q$. These linking numbers can be calculated from the statement of Theorem \ref{pplumb+} and the  linking numbers of knots in $Q$ lying over ovals and one-sided curves in $\rp$  as specified in \cite[Remark (3.1)]{G2}.   In this way,  one may calculate  all the entries in $\vec c$ except $\vec c_{u_1}$ and $\vec c_{u_3}$. 
 
 To calculates $\vec c_{u_3}$, we  need the linking number of $J$, the fiber over $u_2$,  with the fiber over $u_3$. This linking number is $\frac 1 2$, and can be calculated as the $(2,3)$ entry of the inverse of 
 \[ -\left(
\begin{array}{cccc}
 1 & 1 & 1 & 1 \\
 1 & 2 & 0 & 0 \\
 1 & 0 & 2 & 0 \\
 1 & 0 & 0 & 2 \\
\end{array}
\right) .\]
We also use a diffeomorphism of $Q$ which interchanges the fibers over $u_2$ and $u_3$ leaving the fibers over other vertices of $\Gamma$  fixed to conclude the  linking number of the fiber over $u_2$ to a components of  $L(C)$ lying over an oval agree with the linking numbers of $J$ with that component.

Finally one can determine $\vec c_{u_1}$ from $\vec c_{u_2}$, $\vec c_{u_3}$, and $\vec c_{R_1}$ and the fact that $\vec c$ is $p$-characteristic for $\Gamma^+$ at $u_3$ for all $p$. Here we use the last claim made in Theorem \ref{cggl}.
\end{proof}

\begin{proof}[Proof of Proposition \ref{scarceodd}] Since ${c}_\fv \ne 0 \mod{p}$  for the added vertices at arrowheads, neither these nor any vertex indexed by an oval  can be in $\fZ_{\vec c, p}$. This leaves only the vertices indexed by regions or  $u_1$ or $u_3$ as possible vertices for  $\fZ_{\vec c, p}$.
 As $\vec c_{R_1}=1$  where $R_1$ is the outer region, $u_1$ is adjacent to a vertex with ${c}_\fv \ne 0 \mod{p}$, and so is not in $\fZ_{\vec c, p}$.
\end{proof}

\begin{proof}[Proof of Proposition \ref{scarceeven}] Regardless of $\Lambda^- $ and $ \Lambda^+$: since ${c}_\fv \ne 0 \mod{p}$  for the added vertices at arrowheads, neither these nor any vertex indexed by an oval  can be in $\fZ_{\vec c, p}$.  This leaves only the vertices indexed by regions  and $u_1$, $u_2$, $u_3$ as possible vertices for  $\fZ_{\vec c, p}$.

If $\Lambda^- \ne \Lambda^+   \pmod{p}$, by proposition  \ref{ceven},  $\vec c_{u_1}=\vec c_{u_2}=\vec c_{u_3} \ne 0 \pmod{p}$.  Thus
$u_1$, $u_2$,  and $u_3$ do not belong to $\fZ_{\vec c, p}$. As  ${R_1}$ is connected to $u_1$,  neither can   $R_1$  belong to $\fZ_{\vec c, p}$.

If $\Lambda^- = \Lambda^+   \pmod{p}$, by proposition  \ref{ceven}  $\vec c_{u_1}=\vec c_{u_2}=\vec c_{u_3}= \vec c_{R_1} = 0 \pmod{p}$.  The condition $\Pi^+_o -\Pi^-_o =1 \pmod{p}$ is equivalent  to $\vec c_{o} = 0.$ The result follows
\end{proof}

\section{Proof  of Theorem \ref{step} }\label{section.step}

\begin{proof}[Proof  of Theorem \ref{step}]
Note that $\Sign(A_{\fZ_{\vec c, p}})-e_{\vec c, p} -2$ takes a constant value on the set of primes which do not divide any nonzero entry of $\vec c$.
Thus these terms can be replaced in the formula for $\sig_{b/p}(C)$ in Theorem \ref{form} by a constant function on complement in $(0,1/2)$ of the set of points whose denominators are primes dividing some non-zero  entry of $\vec c.$. Again  using Theorem \ref{form}, the nullity $\eta_p(C)$ takes the constant value $\nul$ on the set of  primes that do not divide some non-zero entry of $\vec c$. The remaining terms in the formula for $\sig_{b/p}(C)$  in Theorem \ref{form} can be rewritten
to agree with a piecewise polynomial function of $x=b/p$ whose only discontinuities are at rational points whose denominators are divisors of a nonzero entry of $\vec c$. In fact these remaining terms can be rewritten as 
\[ 2 ((x\vec c^+))({A_{\Ga^+}})((-x \vec c^+))^t+2 x(1-2x)\Delta,\]
Here $(( \vec v))$ denotes the entry by entry  value of the function which assigns to a real number that number minus the greatest integer less than that number.
Thus we have that  $\sig_{b/p}(C)$ extends to  a piecewise 
polynomial
function of $x=b/p$ whose only discontinuities are at rational points whose denominators are divisors of a nonzero entry of $\vec c$. 
On the set of rational points with prime denominators which are  relatively prime to the non-zero entries in $\vec c$, this  function must take integral values.
As this set of points is dense, the function is a step function which takes integral values.
\end{proof}

\appendix

\section {Linking numbers}\label{applk}

\begin{prop}\label{linking}
Let $M$ be a rational homology sphere and $W$ be a 4-manifold with boundary $M$. Suppose $K^{(i)}$ for 
$i \in \{1,2\}$  are knots in $M$, and  $K^{(i)} = \partial F^{(i)}$, where the $F^{(i)}$ are surfaces in $W.$
Let $\{G_j\}$ be  a collection of 2-cycles in the interior of $W$ whose homology classes form a basis for $H_2(W,\BQ)$, and $A$ the matrix for the intersection pairing on $H_2(W,\BQ)$ 
with respect to $\{G_j\}$.  Let $\vec k^{(i)}$ be the vector with entries $k^{(i)}_j= G_j \circ F^{(i)}$, then
\[ \Lk(K^{(1)},K^{(2)}) =  F^{(1)}\circ F^{(2)}- \vec k^{(1)} A^{-1} (\vec {k^{(2)}})^t . \]
\end{prop}

\begin{proof} Let $S^{(i)}$ be  2-chains in $M$ with $\BQ$ coefficients with boundary $K^{(i)}$.
We consider the  two rational $2$-cycles $z^{(i)}= -S^{(i)}+ F^{(i)}$. We have  that $ G_j \circ z^{(i)} = G_j \circ F^{(i)}= k^{(i)}_j$.
If $\vec b^{(i)}$ denotes the vector $A^{-1}\vec k^{(i)}$,  then $z^{(i)}$ is homologous to $\sum {b^{(i)}_j} G_j$.    It follows that
 \[z^{(1)}\circ  z^{(2)}= \vec k^{(1)} A^{-1} (\vec {k^{(2)}})^t.\] On the other hand, it easy to see geometrically that
 \[z^{(1)}\circ  z^{(2)}= F^{(1)}\circ F^{(2)}-\Lk(K^{(1)},K^{(2)}). \]
 See the schematic picture in Figure \ref{inter} where we have pushed $S^{(2)}$ further into $W$ than $S^{(1)}$ with respect to the collar structure for the boundary of $W$.
\end{proof}

\begin{figure}[h]
\includegraphics[width=2in]{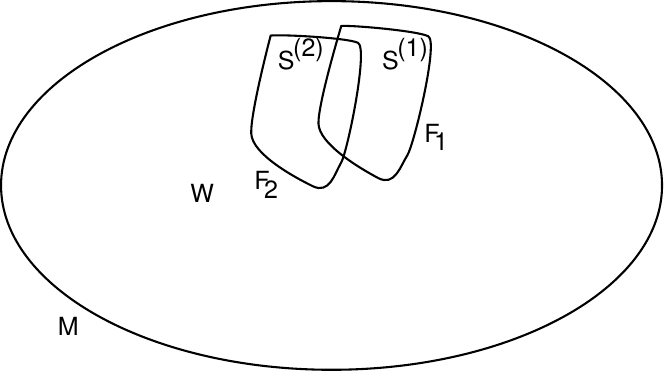} 
\caption{} \label{inter}
\end{figure}

The well-known formula for linking numbers of links in surgery descriptions \cite[Theorem 3.1]{CK} or \cite[Corollary 1.2)]{PY} is a special case of the Proposition \ref{linking}. A further special case is:

\begin{prop}\label{Plumblinking} Suppose a rational homology sphere $M$ is the boundary of plumbing $W_\Gamma$  along a tree $\Gamma$ with associated intersection matrix  $A_\Gamma$. The linking number of the fiber over the $i$th sphere  the fiber over the $j$th  sphere is the $(i,j)$ entry of $-{(A_\Ga)}^{-1}$. This also gives the linking number of two distinct fibers when $i=j$.
\end{prop}

Another special case is when $W$ is a rational ball so that the term $\vec k^{(1)} A^{-1} \vec {k^{(2)}}^t$ vanishes. This was used in \cite{G2} to calculate linking numbers in the projective tangent bundle and circle tangent bundle of $\rp.$

\section{Mathematica Program to compute $\sig_{b/p}(C)$ and $\eta_{p}(C)$}\label{sd} 

\begin{footnotesize}
\begin{lstlisting}
F=Function[T,
  w={1,2,2,2};e={{1,2},{1,3},{1,4}};s={0,0,0,0}; n=4;
  TT = T - (m=s[[2]]=Coefficient[T,J])*J;
  branch[ Pos[TT], 4, 1 ];
  A=Table[0,{n},{n}];    (* A = adjacency matrix of weighted graph Gamma *)
  Do[ A[[i,i]] = w[[i]],{i,n}];
  Do[ i=e[[ie,1]]; j=e[[ie,2]]; A[[i,j]]=A[[j,i]]=1,{ie,Length[e]}];
  cc = c = -2Inverse[A].s;
  AA=Table[0,{m+n},{m+n}]; Do[AA[[i,j]]=A[[i,j]],{i,n},{j,n}];
           (* AA = adjacency matrix of weighted graph Gamma+ *)
  nn=n; ee=e;
  Do[If[Not[s[[i]]==0],
    AppendTo[cc,2s[[i]]]; nn++;AA[[i,nn]]=AA[[nn,i]]=1;
     AppendTo[ee,{i,nn}]], {i,n}];
  De = 2s.Inverse[A].s;
  h = Mod[cc,p];
  X={};  (* X will be the complement of Z_p in Gamma+  *)
  zp=0;
  Do[
    If[ h[[i]]==0, zp++;
      Do[
        If[j==i,Continue[]];
        If[AA[[i,j]]==0,Continue[]];
        If[h[[j]]==0,Continue[]];
        AppendTo[X,{i}];Break[], (*j<>i & h[[j]]<>0 & [i,j] in e*)
      {j,nn}],
      AppendTo[X,{i}]   (* h[[i]].ne.0, hence i not in Z_p *)
    ],
  {i,nn}];
  n0 = nn - Length[X];          (* n0 = |v(Z_p)| *)
  A0 = If[n0==0, {}, Delete[Transpose[Delete[AA,X]],X]]; (*A0 = A_{Z_p}*)
  e0 = 0;  (* e0 = e_p *) (* E_p=1-nn+e0, as Gamma+ is a tree *)
  Do[ If[ Not[ h[[ee[[ie,1]]]]*h[[ee[[ie,2]]]]==0 ], e0++ ], {ie,Length[ee]}];
  mu = Mod[b*h,p].AA.Mod[-b*h,p];
  sn = SiNu[A]; sn0 = If[n0==0,{0,0},SiNu[A0]];
  sigbp = 2/p^2(mu+b(p-2b)De) + sn0[[1]] -e0 - 2;
  etap = sn0[[2]] + n0 - 2zp + nn - e0 - 1;
  {sigbp,etap}
];
SiNu=Function[A,Module[{ev,i,sg=0,nu=0},
  ev=Re[Eigenvalues[A]];
  Do[If[ Abs[ev[[i]]]<10.^(-7), nu++, sg += Sign[ev[[i]]] ],  
   {i,Length[ev]}];{sg,nu}
]];

branch = Function[{R,k,d},Module[{T,t,fl,i,j,len},
  If[ Not[AtomQ[R]],
    T=R[[1]]+zero; len=Length[T];
    Do[
      If[IntegerQ[T[[i]]-zero],Continue[]];
      fl=FactorList[T[[i]]]; t=fl[[2,1]];
      Do[
        e=Join[e,{{k,n+1},{n+1,n+2}}]; w[[k]]-=2;
        sg=If[AtomQ[t],t,t[[0]]];
        w=Join[w,{0,2}]; s=Join[s,{If[IntegerQ[sg-Pos],1,-1](-1)^(d),0}];
        n+=2; m++; branch[t,n,d+1],
      {j,fl[[1,1]]}],
    {i,len}];
  ];
]]
\end{lstlisting}
\end{footnotesize}

 The  complex scheme $\left< J \amalg 1^- \left<2^-\right> \amalg 2^+ \right>$, for instance is inputted as
  $J+\text{Neg}[2 \text{Neg}]+2 \text{Pos}$. Here is a sample computation:
 
In[1]= p=7; b=2; F[J+\text{Neg}[2 \text{Neg}]+2 \text{Pos}]

Out[1]= \{-10,1\}
\section {Behavior  of $\sig$ and $\nul$ for schemes of even type }\label{even1/4}

\begin{thm} If $C$ is a non-empty complex scheme of even type, then
$$\sig_x(C)= \sig_{\frac 1 2 -x}(C).$$
Thus $\sig_x(C)$ does not have a discontinuity at $x=\frac 1 4$.
 \end{thm}
 \begin{proof}  We note that $L(C)$ is null homologous, so there is a Seifert surface for $L$ and a corresponding Seifert matrix $V$. We have  the 
ordinary signature function  $\sigma_\omega $ given by the signature of $V_\omega$ where  $V_\omega= (1-\omega)V+ (1- \omega^{-1})V^T$, and $\omega$ lies on the unit circle. 
We have that $\sigma_\omega = \sigma_{\omega^{-1}}$.
 Let $\sigma'_\omega$ denote $\sigma_\omega$ redefined at jumps to be the average of the one sided limits
 except when jumps occur at  $p$th roots of unity, where $p$ is an odd prime. 
  Moreover $\sigma_{a/p}(L(C))= \sigma_{e^{2 \pi i a/p}}.$ As points of the form $e^{2 \pi i a/p}$ are dense and using the  re-parameterization in the definition of $\sig_{b/p}$, we have
$\sig_{x}(C) = \sigma'_{e^{4 \pi i x}}L(C))$. Thus
$$\sig_{\frac 1 2 -x}(C) = \sigma'_{e^{2 \pi i -4 \pi i x}}(L(C))= \sigma'_{e^{-4 \pi i x}}(L(C))=\sigma'_{e^{4 \pi i x}}(L(C))=\sig_{x}(C) .$$
Thus $\sig_x(C)$ is symmetric around $1/4$, and so will be continuous at $\frac 1 4.$
\end{proof}

If $C$ is a complex scheme  of even type, let $\sig(C)$ denote $\sig_{1/4}(C)$.  
We let $n(C)$  denote the number of ovals of $C$ with odd parity. These are called odd ovals.
The number of odd ovals which are empty is denoted  $n_+(C)$.
The number of odd ovals which have one oval immediately inside them is denoted  $n_0(C)$.
The number of odd ovals which have more than one oval immediately inside them is denoted   $n_-(C)$. Note that the subscripts reflect the Euler characteristic of the region bounding the oval from the inside.
As usual, we drop the $C$ from the notation for numerical characteristics of $C$ when it should cause no confusion.

\begin{thm} Suppose $C$ is a non-empty complex scheme of even type.
If $l$ is even, 
 $$|\sig+ n+n_0+2n_ - -2(\Pi_+-\Pi_-)|+ \nul \le 
 l-n_+ -n_- -1.$$
 If  $l$ is odd and $C$ has just one outer oval,
  $$|\sig+ n+n_0+2n_ - -2(\Pi_+-\Pi_-)|+ \nul \le
 l-n_+ -n_-+1 . 
 $$
 If $l$ is odd and $C$ has more than one outer oval, 
 $$|\sig+ n+n_0+2n_ - -2(\Pi_+-\Pi_-)+1|+ \nul \le    l-n_+ -n_-
 .$$
 Also
 $$\sig + \nul \equiv l -1 \pmod{2}.$$
 \end{thm}
 \begin{proof}
We have that $\sig$ is the limit of $\si_\omega$ on the unit circle as $ \omega$ approaches 
$-1$. $\nul$ is the $\eta_\omega$ for $\omega$ near $-1$. 

In \cite[sections  4 and 5]{G5}, a  signature \cite{G3} for links, denoted  $S_0(L)$,  and a nullity (denoted    $ \eta_0(L)$ ) is computed   for $L=L(C)$. 
Here the subscript $0$ refers to the zero homology class of $H_1(Q)$.
We have \begin{equation} \label{sa} \sigma_{-1}(L(C))= S_0(L)-l+2(\Pi_+ -\Pi_-)  \end{equation} 
 and 
\begin{equation} \label{na} \eta_{-1}(L(C))= \eta_0(L).   \end{equation}
Above, the $-l$ is to correct for the choice of framing and we  subtract the total linking number of  $L(C)$ which is $2(\Pi_- -\Pi_+)$ to convert between  signature conventions using  oriented spanning surfaces links and unoriented spanning surfaces.  The difference between these signatures is well explained in \cite[p63]{GL}.
Using \cite[Prop 4.2 , 5.1]{G5},  \ref{sa} and \ref{na}, we write out $\sigma_{-1}(L(C))$ and $\eta_{-1}(L(C))$.
Thus, if  $l$ is even, 
$$\sigma_{-1}(L(C))= 2(\Pi_+-\Pi_-)- n -n_0-2n_- $$
and 
$$ \eta_{-1}(L(C))=l -n + n_0 -1.$$
If $l$ is odd, and $C$ has just one outer oval, $$\sigma_{-1}(L(C))=  2(\Pi_+-\Pi_-) -n -n_0-2n_-$$ and $\eta_{-1}(L(C))= l-n+ n_0+1$.
If $l$ is odd, and $C$ has more than one outer oval, $$\sigma_{-1}(L(C))=  2(\Pi_+-\Pi_-)-n-n_0-2n_- -1 $$ with $\eta_{-1}(L(C))=  l-n +n_0$.

Now consider the eigenvalues of the continuous family of Hermitian matrices $V_\omega$. The non-zero eigenvalues at $\omega=-1$ must keep their signs  as $\omega$ is perturbed from $-1$  but the zero eigenvalues can suddenly become non-zero but $\nul$ of these zero eigenvalues must remain unchanged. This follows from the continuity of the characteristic polynomial of $V_\omega$ and the continuity of the roots of a continuous family of polynomials
\cite[Appendix V]{W}.
Finally the change in the nullity must be congruent modulo two to the change in
the signature as we move away from $-1$.
\end{proof}


\begin{thebibliography}{HNK}

\bibitem[CG]{CG} { \sc A.~J.~Casson and  C. McA. Gordon}. Cobordism of classical knots. {\em   Progr. Math.}{\bf  62} `\'A la recherche de la topologie perdue,' 181--199, Birkh\"auser Boston, Boston, MA, 1986.

\bibitem[CK]{CK} { \sc J.C.~Cha, K.H.~ Ko}. Signatures of links in rational homology spheres.
{ \em Topology} {\bf 41} (2002), no. 6, 1161--1182. 

\bibitem[EN]{EN}{\sc D. Eisenbud, W. Neumann.} Three-dimensional link theory and invariants of plane curve singularities. {\em Annals of Mathematics Studies,} {\bf 110}. Princeton University Press, Princeton, NJ, 1985.

\bibitem[F]{F}{\sc S. Fiedler-Le Touz\'e,} M-curves of degree 9 or 11 with one unique non-empty oval.  arXiv:1412.5313
 

 \bibitem[G1]{G1} {\sc P. Gilmer.}  Configuration of surfaces in 4-manifolds. {\em  
 Trans. Amer. Math. Soc} {\bf 264} (1981) 353-380 
 
\bibitem[G2]{G2} {\sc P. Gilmer.} Real algebraic
 curves and link cobordism. { \em Pacific J. Math.} {\bf 153} (1992)
 31--69 

 \bibitem[G3]{G3} {\sc  P. Gilmer.}  Link cobordism in rational homology
 3-spheres. {\em Journal of Knot theory and its Ramifications}
 {\bf 2} (1993)  285-320

 \bibitem[G4]{G4} {\sc  P. Gilmer.}  Signatures of singular branched covers. {\em  Math. Ann.}{\bf  295}  (1993),  no. 4, 643--659.

\bibitem[G5]{G5} {\sc  P. Gilmer.}   Real algebraic curves and link cobordism II.  Topology of real algebraic varieties and related topics,  73--84,{\em  Amer. Math. Soc. Transl. Ser. 2},{\bf 173}, Amer. Math. Soc., Providence, RI, 1996.

\bibitem[GL]{GL} Gordon, C. McA.; Litherland, R. A. On the signature of a link. Invent. Math. 47 (1978), no. 1, 53Ð69

\bibitem[HNK]{HNK} {\sc F. Hirzebruch, W. Neumann, S. Koh.} Differentiable manifolds and quadratic forms. { \em Lecture Notes in Pure and Applied Mathematics} {\bf  4}. Marcel Dekker, Inc., New York, 1971

\bibitem[Ka]{Ka}{ \sc L;Kauffman}. On knots. Annals of Mathematics Studies, 115. Princeton University Press, Princeton, NJ, 1987. 

\bibitem[K]{K} {\sc R. Kirby.} A calculus for framed links in $S^{3}$. {\em Invent. Math.} {\bf  45}  (1978), 35--56.

\bibitem[O]{O} {\sc O. Viro}, Branched coverings of manifolds with boundary, and invariants of links. I. (Russian) 
Izv. Akad. Nauk SSSR Ser. Mat. 37 (1973), 1241Ð1258.
   {English translation: Math. USSR-Izv. 7 (1973), no. 6, 1239Ð1256 (1975)}

\bibitem[PY]{PY} {\sc J.Przytycki, A.Yasuhara.}
Linking numbers in rational homology 3-spheres, cyclic branched covers and infinite cyclic covers.
{\em  Trans. Amer. Math. Soc.} {\bf  356}  (2004),  3669--3685. 

\bibitem[R]{R} {\sc  V.A.Rokhlin.} Complex Topological Characteristics.
{ \em Uspeki Mat. Nauk.} {\bf 35 } (1978) 77-89
(Russian)
English
transl.  { \em Russian Math Surveys} {\bf 33 } (1978) 85-98 

\bibitem[TW]{TW} {\sc E.Thomas, J.Wood.}
On manifolds representing homology classes in codimension $2$.
{ \em Invent. Math.} {\bf 25} (1974), 63--89. 

\bibitem[TV]{TV} {\sc V. Turaev,   O. Ya Viro.} Estimates of twisted homology. { \em Proc. of VII union Topology Conference, Minsk} (1977)

\bibitem[V1]{V} {\sc   O. Ya Viro.} Progress in the topology of real
algebraic
varieties over the last six years.{ \em Uspehi Mat. Nauk.} {\bf 41} (1986)
 45-67 (Russian)
English
transl.  {\em Russian Math. Surveys}{ \bf
41} (1986) 55--82

\bibitem[V2]{V2} {\sc   O. Ya Viro.} Twisted acyclicity of a circle and signatures of a link. {\em J. Knot Theory Ramifications} {\bf  18}  (2009),  no. 6, 729--755. 

\bibitem[W]{W} {\sc H. Whitney.} Complex Analytic Varieties.  Adison-Wesley, Reading, 1972



\end{thebibliography}
 \end{document}